\documentclass{ijclclp}
\usepackage{physics}
\title{Differentiation theorems for BV functions of several variables, and applications}
\author{Xianrui Zhang\textsuperscript{1}}
\affiliation{Department of Mathematics, Sun Yat-sen University, Guangzhou, China}
\email{\texttt{zhangxr86@mail2.sysu.edu.cn; zhangxianrui@gmdyxx.wecom.work}}

\begin{document}

\maketitle
\thispagestyle{firstpage}
\begin{abstract} 
We extend the classical Lebesgue and Fubini differentiation theorems to functions of several variables, using the notions of joint derivative and joint monotonicity. Our first main result shows that for a function $f$ of bounded variation, the joint derivative exists almost everywhere, its $L^1$ norm is bounded by the total variation of $f$, and equality in this bound characterizes absolute continuity. Our second main result shows that, for a convergent series of jointly monotone increasing functions, the joint derivative of the sum equals the sum of the joint derivatives almost everywhere.
\end{abstract} 

\begin{keywords}
multivariable calculus, bounded variation, Lebesgue differentiation theorem, Fubini differentiation theorem
\end{keywords}

\msc{26B30, 26B35}

\section{Introduction}
In this section, we recall the notions of \textit{joint increment}, \textit{joint monotonicity}, \textit{joint derivative}, \textit{bounded variation} (in the sense of Hardy and Krause), and \textit{absolute continuity} for functions of several variables. We then present the main results of this work, including a Jordan decomposition theorem, a Lebesgue differentiation theorem for functions of bounded variation, and a Fubini differentiation theorem for series of jointly monotone increasing functions. 

We begin with the definition of joint increment. 

\begin{definition}{(joint increment, cf. [5, \S 254])} 
    Let \(f\) be a function of $n$ variables. When \(n=1\), define for $a, b\in\mathbb R$, 
    \[\Delta^1_{f}[a,b]=f(b)-f(a);\]
    when \(n\ge 2\), define for \(\boldsymbol{a}=(a_{1},a_{2},\dots,a_{n})\), 
    \(\boldsymbol{b}=(b_{1},b_{2},\dots,b_{n})\in\mathbb R^n\), 
    \begin{equation}
        \begin{aligned}
            \Delta_{f}^{n}[\boldsymbol{a},\boldsymbol{b}]&=\Delta_{f|_{x_n=b_n}}^{n-1}[(a_{1},\cdots, a_{n-1}),(b_{1},\cdots,b_{n-1})]\\
            &-\Delta_{f|_{x_n=a_n}}^{n-1}[(a_{1},\cdots, a_{n-1}),(b_{1},\cdots,b_{n-1})].            
        \end{aligned}
    \end{equation}
\end{definition}
\begin{sign}
    For \(\boldsymbol{a}=(a_{1},a_{2},\dots,a_{n})\), 
    \(\boldsymbol{b}=(b_{1},b_{2},\dots,b_{n})\in\mathbb R^n\), we use \([\boldsymbol{a},\boldsymbol{b}]\) to denote the $n$-dimensional rectangle with diagonal vertices $\boldsymbol{a}$ and $\boldsymbol{b}$. For \(\boldsymbol{\epsilon}=(\epsilon_{1},\epsilon_{2},\dots,\epsilon_{n})\), we use \(\boldsymbol{\epsilon}\circ(\boldsymbol{b}-\boldsymbol{a})\) to denote 
    \[\boldsymbol{\epsilon}\circ(\boldsymbol{b}-\boldsymbol{a})=\big(\epsilon_1(b_1-a_1),\dots,\epsilon_n(b_n-a_n)\big).\]
\end{sign} 
\begin{remark}
     (1) is equivalent to
     \begin{equation}
         \Delta_{f}^{n}[\boldsymbol{a},\boldsymbol{b}]=\sum_{\epsilon_1,\dots,\epsilon_n\in\{0,1\}}(-1)^{n+\epsilon_1+\dots+\epsilon_n}f\big(\boldsymbol{a}+\boldsymbol{\epsilon}\circ(\boldsymbol{b}-\boldsymbol{a})\big).
     \end{equation}
\end{remark}

The definition of joint derivative is based on that of joint increment. 

\begin{definition}{(joint derivative)}
    Suppose \(f\) is defined in a neighborhood of \(\boldsymbol{x}=(x_1, x_2, \dots, x_n)\in\mathbb R^n\). The joint derivative of \(f\) at \(\boldsymbol{x}\) is defined by 
    \[f^{(n)}(\boldsymbol{x}):=\lim_{h\to 0^+}\frac{\Delta_{f}^{n}[x_1,x_1+h]\times[x_2,x_2+h]\times\dots\times[x_n,x_n+h]}{h^n},\]
    when the limit exists. \\[-0.5em]
\end{definition}

Similarly, joint monotonicity is defined in terms of joint increment. In what follows, we use  \((x_1, x_2, \dots, x_n)\leq(y_1, y_2, \dots, y_n)\) to mean that \(x_i\leq y_i\) holds for all $i=1,\dots,n$.


\begin{definition}{(joint monotonicity, cf. [5, \S 255])}
    Let \(f\) be a function defined on a rectangle \([\boldsymbol{a},\boldsymbol{b}]\subset\mathbb R^n\). If for any \(\boldsymbol{x}, \boldsymbol{y}\in [\boldsymbol{a},\boldsymbol{b}]\) with \(\boldsymbol{x}\leq\boldsymbol{y}\), we have \(\Delta_{f}^{n}[\boldsymbol{x},\boldsymbol{y}] \geq 0 \), then \(f\) is called a jointly monotone increasing function on \([\boldsymbol{a},\boldsymbol{b}]\). \\[-0.5em]
\end{definition}

We now recall the notion of bounded variation in the sense of Hardy and Krause (cf. [1]). 

\begin{definition}{(bounded variation)}
    Let \(f\) be a function defined on the rectangle \([\boldsymbol{a},\boldsymbol{b}]=[a_1,b_1]\times\dots\times[a_n, b_n]\subset\mathbb R^n\). If over all finite partitions
    \[\Delta:a_i =x_{i0}<x_{i1}<\dots<x_{iN}=b_i\quad(i=1,\dots,n),\] 
    denoting by \(\{I_i\}_{i=1}^{N^n}\) the sub-rectangles generated by $\Delta$, we have
    \begin{equation}
            \bigvee\limits_{\boldsymbol{a}}^{\boldsymbol{b}}(f):=\sup_{\Delta}\sum_{i=1}^{N^n}|\Delta_f^n I_i|<+\infty, 
    \end{equation}
    then \( f\) is said to be of bounded variation on \([\boldsymbol{a},\boldsymbol{b}]\), denoted by \(f\in BV[\boldsymbol{a},\boldsymbol{b}].\) 
    The supremum \(\bigvee\limits_{\boldsymbol{a}}^{\boldsymbol{b}}(f)\) is called the total variation of \(f\) on \([\boldsymbol{a},\boldsymbol{b}]\).\\[-2em]
\end{definition}
\begin{figure}[h!]
    \centering
    \includegraphics[width=8cm]{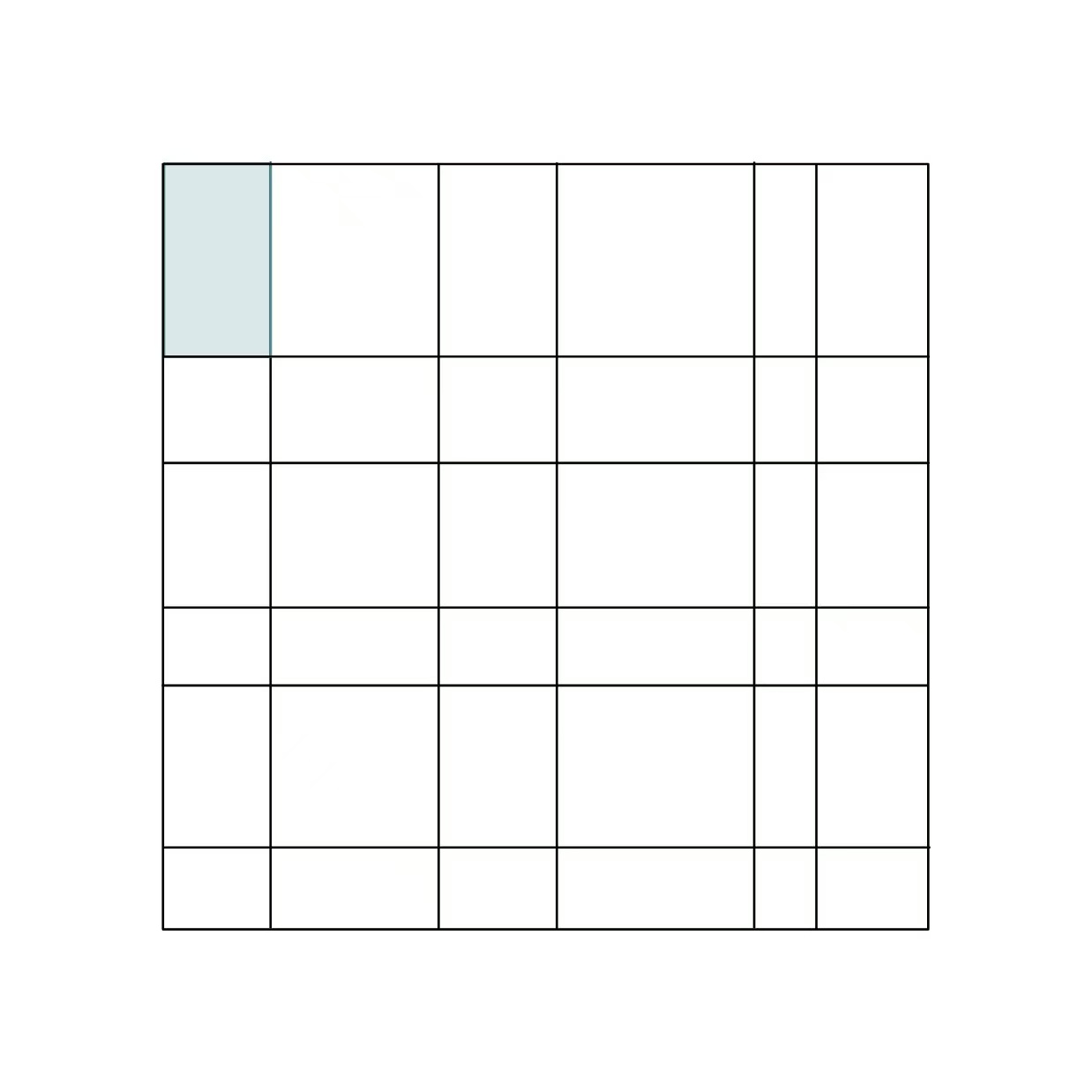}\\[-2em]
    \caption{An illustration of Definition 1.4 in two dimensions. The sum of the absolute values of the joint increments of \(f\) over the sub-rectangles is uniformly bounded.}
\end{figure}

In what follows, we use $m(\cdot)$ to denote the $n$-dimensional Lebesgue measure.

\begin{definition}{(absolute continuity, cf. [5, \S 254])}
    Let \(f\) be a function defined on a rectangle \([\boldsymbol{a},\boldsymbol{b}]\subset\mathbb R^n\). If for any \(\varepsilon>0\), there exists a \(\delta>0\), such that for any finite collection of sub-rectangles \(\{I_i\subset [\boldsymbol{a},\boldsymbol{b}]\}_{i=1}^N\) with \(\mathring{I}_i\cap\mathring{I}_j=\varnothing\ (i\neq j)\), and satisfying
    \[\sum_{i=1}^N m(I_i)<\delta,\]
    we have
    \[\sum_{i=1}^N \big\lvert\Delta^n_f I_i\big\rvert<\varepsilon,\]
    then \(f\) is said to be absolutely continuous on \([\boldsymbol{a},\boldsymbol{b}]\), denoted by \(f\in AC[\boldsymbol{a},\boldsymbol{b}]\). \\[-0.5em]
\end{definition}
\begin{figure}[h!]
    \centering
   \includegraphics[width=8cm]{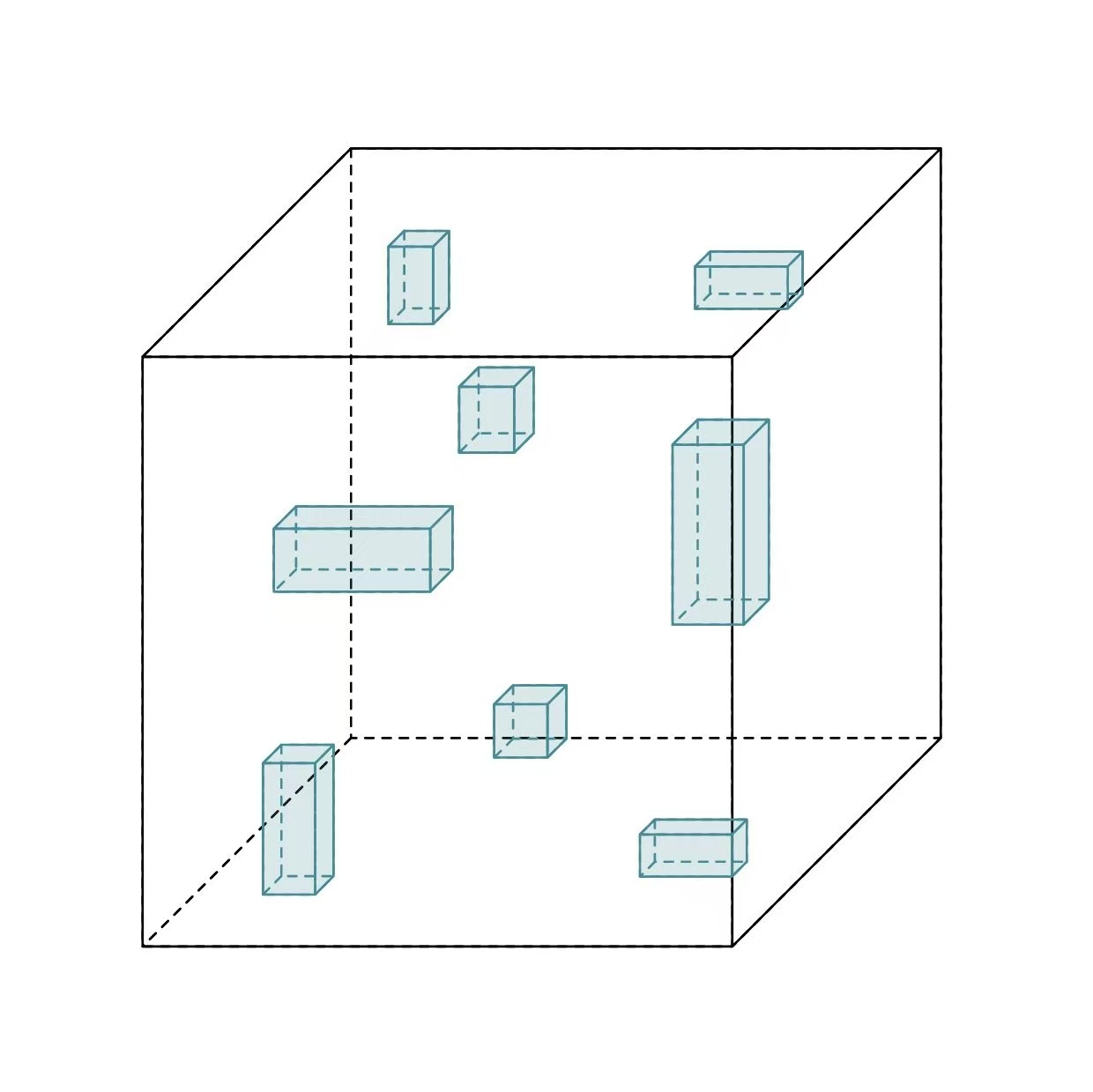}\\[-1.5em]
   \caption{An illustration of Definition 1.5 in three dimensions. Provided that the total volume of the shaded sub-rectangles is sufficiently small, the sum of the absolute values of the joint increments of the function over these sub-rectangles is sufficiently small.} 
\end{figure}

We now state the main results of this work. The proofs are given in Section 2. 

\begin{theorem}{(Lebesgue differentiation theorem for jointly monotone functions, cf. [6])} 
    Let \(f\) be a jointly monotone increasing function on a rectangle \([\boldsymbol{a},\boldsymbol{b}]\subset\mathbb R^n\). Then the joint derivative \(f^{(n)}\) exists almost everywhere in \([\boldsymbol{a},\boldsymbol{b}]\), and satisfies 
    \begin{equation}
        \int_{\boldsymbol{a}}^{\boldsymbol{b}}f^{(n)}(\boldsymbol{x})\diff \boldsymbol{x}\leq\Delta_f^n[\boldsymbol{a},\boldsymbol{b}].
    \end{equation}
\end{theorem}

Corollary 1.3 below shows that equality in (4) holds if and only if \(f\) is absolutely continuous. 

It is implied in Theorem 1.1 that if \(f\) is jointly increasing on $[\boldsymbol{a},\boldsymbol{b}]$, then \(f^{(n)}\) is nonnegative, Lebesgue measurable, and belongs to \(L[\boldsymbol{a},\boldsymbol{b}]\) (the class of Lebesgue integrable functions on $[\boldsymbol{a},\boldsymbol{b}]$). 

The joint derivative can also be considered in other quadrants. This leads to essentially the same definition as Definition 1.2. 

\begin{definition}{(joint derivative in the \(\boldsymbol{\epsilon}\)-th quadrant)}
    Suppose \(f\) is defined in a neighborhood of \(\boldsymbol{x}\in\mathbb R^n\). For \(\boldsymbol{\epsilon}=(\epsilon_1,\dots,\epsilon_n)\), where \(\epsilon_i\in\{-1, 1\}\), define 
    \[D_{\boldsymbol{\epsilon}}f(\boldsymbol{x}):=\lim_{h\to 0^+} \frac{\Delta_f^n[\boldsymbol{x}, \boldsymbol{x}+h\boldsymbol{\epsilon}]}{h^n}\]
    when the limit exists. \\[-0.5em]
\end{definition}

Note that \([x_i,x_i+\epsilon_i h]\) is understood as \([x_i+\epsilon_i h,x_i]\) when \(\epsilon_i=-1\).
    
\begin{proposition}
    Suppose \(f\) is a jointly monotone increasing function on a rectangle \([\boldsymbol{a},\boldsymbol{b}]\subset\mathbb R^n\). Then for any \(\boldsymbol{\epsilon}\) as above, \(D_{\boldsymbol{\epsilon}}f(\boldsymbol{x})\) exists almost everywhere on \([\boldsymbol{a},\boldsymbol{b}]\). Moreover, for any \(\boldsymbol{\epsilon}\neq\boldsymbol{\epsilon}'\), we have
    \begin{equation}
        D_{\boldsymbol{\epsilon}}f(\boldsymbol{x})=D_{\boldsymbol{\epsilon'}}f(\boldsymbol{x}),\quad \text{a.e. }\boldsymbol{x}\in[\boldsymbol{a},\boldsymbol{b}].
    \end{equation}
\end{proposition}

As an application of Theorem 1.1, we have the following corollary, which can be used to characterize when a jointly monotone function is absolutely continuous. 

\begin{corollary}
    Let \(f\) be a jointly monotone increasing function on a rectangle \([\boldsymbol{a},\boldsymbol{b}]\subset\mathbb R^n\). Then \(f\in AC[\boldsymbol{a},\boldsymbol{b}]\) if and only if
    \[\int_{\boldsymbol{a}}^{\boldsymbol{b}}f^{(n)}(\boldsymbol{x})\diff\boldsymbol{x}=\Delta_f^n[\boldsymbol{a},\boldsymbol{b}].\]
\end{corollary}

Note that the last condition only requires the equality to hold on \([\boldsymbol{a},\boldsymbol{b}]\); it will automatically imply that the same equality holds on any sub-rectangle \([\boldsymbol{c},\boldsymbol{d}]\subset[\boldsymbol{a},\boldsymbol{b}]\), that is, 
    \[\int_{\boldsymbol{c}}^{\boldsymbol{d}}f^{(n)}(\boldsymbol{x})\diff\boldsymbol{x}=\Delta_f^n[\boldsymbol{c},\boldsymbol{d}].\]

In order to extend Theorem 1.1 to functions of bounded variations, we will use the following Jordan decomposition theorem (for the two-dimensional case, see [1, Theorem 5]). 

\begin{theorem}{(Jordan decomposition theorem)} Let $f$ be a function defined on a rectangle \([\boldsymbol{a},\boldsymbol{b}]\subset\mathbb R^n\). Then 
     \(f\in BV[\boldsymbol{a},\boldsymbol{b}]\) if and only if there exist two jointly monotone increasing functions \(g\) and \(h\) on \([\boldsymbol{a},\boldsymbol{b}]\), such that for any \(\boldsymbol{x}\in[\boldsymbol{a},\boldsymbol{b}]\), 
    \begin{equation}
      f(\boldsymbol{x})=g(\boldsymbol{x})-h(\boldsymbol{x}).  
    \end{equation}
\end{theorem}

In view of Theorem 1.1 and Proposition 1.2, an immediate consequence of Theorem 1.4 is that if \(f\in BV[\boldsymbol{a},\boldsymbol{b}]\), then \(f^{(n)}\) exists almost everywhere in \([\boldsymbol{a},\boldsymbol{b}]\); moreover, the joint derivatives of \(f\) in all quadrants exist and coincide almost everywhere. 

We now state the first main result of this work. 

\begin{theorem}{(Lebesgue differentiation theorem for BV functions)}
    Let \([\boldsymbol{a},\boldsymbol{b}]\subset\mathbb R^n\) be a rectangle and let \(f\in BV[\boldsymbol{a},\boldsymbol{b}]\). Then the joint derivative \(f^{(n)}\) exists almost everywhere in \([\boldsymbol{a},\boldsymbol{b}]\), and satisfies 
    \begin{equation}
        \int_{\boldsymbol{a}}^{\boldsymbol{b}}\lvert f^{(n)}(\boldsymbol{x})\rvert\diff \boldsymbol{x}\leq\bigvee_{\boldsymbol{a}}^{\boldsymbol{b}}(f).
    \end{equation}
\end{theorem}

Similar to Theorem 1.1, equality in (7) holds if and only if \(f\in AC[\boldsymbol{a},\boldsymbol{b}]\). We state this as a corollary. 

\begin{corollary}
    Let \([\boldsymbol{a},\boldsymbol{b}]\subset\mathbb R^n\) be a rectangle and let \(f\in BV[\boldsymbol{a},\boldsymbol{b}]\). Then \(f\in AC[\boldsymbol{a},\boldsymbol{b}]\) if and only if
    \[
        \int_{\boldsymbol{a}}^{\boldsymbol{b}}\lvert f^{(n)}(\boldsymbol{x})\rvert\diff \boldsymbol{x}=\bigvee_{\boldsymbol{a}}^{\boldsymbol{b}}(f).
    \]
\end{corollary}

As an application and extension of Theorem 1.1, our second main result extends the classical Fubini differentiation theorem (cf. [4]) to series of jointly monotone increasing functions. 

\begin{theorem}{(Fubini differentiation theorem for jointly monotone functions)}
    Let \(\{f_i\}_{i=1}^\infty\) be a sequence of jointly monotone increasing functions on a rectangle \([\boldsymbol{a},\boldsymbol{b}]\subset\mathbb R^n\). Suppose for every \(\boldsymbol{x}\in[\boldsymbol{a},\boldsymbol{b}]\), the series \(\sum\limits_{i=1}^{\infty}f_{i}(\boldsymbol{x})\) converges. Then 
    \[\left(\sum_{i=1}^{\infty} f_{i}\right)^{\hspace{-0.4em}(n)}\hspace{-0.4em}(\boldsymbol{x})= \sum_{i=1}^{\infty} f_{i}^{(n)}(\boldsymbol{x}),\quad\text{a.e. }\boldsymbol{x}\in[\boldsymbol{a},\boldsymbol{b}].\]
\end{theorem}

Finally, using Theorem 1.1 and Theorem 1.4, simplified proofs for the following classical theorems can be obtained. The proofs can be found in the Appendix. 

\begin{corollary}{(Lebesgue differentiation theorem, cf. [6])}
    Let \([\boldsymbol{a},\boldsymbol{b}]\subset\mathbb R^n\) be a rectangle and let \(f\in L[\boldsymbol{a},\boldsymbol{b}]\). Then 
    \begin{equation}
        \left(\int_{\boldsymbol{a}}^{\boldsymbol{x}}f(\boldsymbol{u})\diff \boldsymbol{u}\right)^{\hspace{-0.2em}(n)}\hspace{-0.4em}(\boldsymbol{x})=f(\boldsymbol{x}),\quad\text{a.e. }\boldsymbol{x}\in[\boldsymbol{a},\boldsymbol{b}].
    \end{equation}
\end{corollary}
\begin{corollary}{(Fundamental Theorem of Calculus, cf. [8], [3])}
    Let \([\boldsymbol{a},\boldsymbol{b}]\subset\mathbb R^n\) be a rectangle and let \(f\in AC[\boldsymbol{a},\boldsymbol{b}]\). Then 
    \begin{equation}
        \Delta_f^n[\boldsymbol{a},\boldsymbol{b}]=\int_{\boldsymbol{a}}^{\boldsymbol{b}}f^{(n)}(\boldsymbol{x})\diff \boldsymbol{x}.
    \end{equation}
\end{corollary}

\section{Proofs of the Theorems}
\subsection{Proof of Theorem 1.1}
To give a proof of Theorem 1.1, we first recall the following. 

\begin{definition}{(componentwise monotone function, cf. [5, \S 255])}
    Let \(f\) be a function defined on \([\boldsymbol{a},\boldsymbol{b}]\). If for any \(\boldsymbol{x},\boldsymbol{y}\in[\boldsymbol{a},\boldsymbol{b}]\) with $\boldsymbol{x}\leq\boldsymbol{y}$, we have  \(f(\boldsymbol{x})\leq f(\boldsymbol{y})\), then \(f\) is called a componentwise monotone increasing function on \([\boldsymbol{a},\boldsymbol{b}]\). 
\end{definition}
\begin{remark}
    An equivalent definition is as follows: for every \(\boldsymbol{x}\in[\boldsymbol{a},\boldsymbol{b}]\), every \(i\), and every $c, d\in[a_i, b_i]$ with \(c\leq d\), 
    \[f(x_1,\dots,x_{i-1},c,x_{i+1},\dots,x_n)\leq f(x_1,\dots,x_{i-1},d,x_{i+1},\dots,x_n).\]
\end{remark}

\begin{lemma}{(for the two-dimensional case, see [1, Theorem 9])}
    Suppose \(f\) is a jointly monotone increasing function on \([\boldsymbol{a},\boldsymbol{b}]\). Define for \(\boldsymbol{x}\in[\boldsymbol{a},\boldsymbol{b}]\), 
    \begin{equation}
        \begin{aligned}[b]
            \widetilde{f}(\boldsymbol{x})
            :&=\Delta_f^n[\boldsymbol{a},\boldsymbol{x}]\\
            &=\sum_{\epsilon_1,\dots,\epsilon_n\in\{0,1\}}(-1)^{n+\epsilon_1+\dots+\epsilon_n}
            f\big(\boldsymbol{a}+\boldsymbol{\epsilon}\circ(\boldsymbol{x}-\boldsymbol{a})\big). 
        \end{aligned}
    \end{equation}
    Then \(\widetilde{f}\) is a componentwise monotone increasing function on \([\boldsymbol{a},\boldsymbol{b}]\).\\[-0.5em]
\end{lemma}
\begin{proof}
    We only give the proof for the case \(\boldsymbol{a}\leq\boldsymbol{b}\). The case where there exists \(b_i<a_i\) can be proved similarly. 

    
    We now prove that for \(\forall\boldsymbol{c}\leq\boldsymbol{d}\in[\boldsymbol{a},\boldsymbol{b}]\), if there exists \(i\) such that \(c_i=d_i\), then
    \begin{equation}
        \Delta_{f}^n[\boldsymbol{c},\boldsymbol{d}]=0.
    \end{equation}
    The left side of (11) expands to
    \begin{equation}
        \begin{aligned}[b]
        &\Delta_{f}^n[c_1,d_1]\times\dots[c_{i-1},d_{i-1}]\times[c_i,c_i]\times[c_{i+1},d_{i+1}]\times[c_n,d_n]\\
        &=\sum_{\epsilon_1,\dots,\epsilon_{i-1},\epsilon_{i+1},\dots,\epsilon_n\in\{0,1\}}(-1)^{n+\epsilon_1+\dots+\epsilon_{i-1}+0+\epsilon_{i+1}+\dots+\epsilon_n}f\big(\boldsymbol{c}+\boldsymbol{\epsilon}|_{\epsilon_i=0}\circ(\boldsymbol{d}-\boldsymbol{c})\big)\\
        &+\sum_{\epsilon_1,\dots,\epsilon_{i-1},\epsilon_{i+1},\dots,\epsilon_n\in\{0,1\}}(-1)^{n+\epsilon_1+\dots+\epsilon_{i-1}+1+\epsilon_{i+1}+\dots+\epsilon_n}f\big(\boldsymbol{c}+\boldsymbol{\epsilon}|_{\epsilon_i=1}\circ(\boldsymbol{d}-\boldsymbol{c})\big).
        \end{aligned}
    \end{equation}
    Because
    \begin{equation}
        f\big(\boldsymbol{c}+\boldsymbol{\epsilon}|_{\epsilon_i=0}\circ(\boldsymbol{d}-\boldsymbol{c})\big)=f\big(\boldsymbol{c}+\boldsymbol{\epsilon}|_{\epsilon_i=1}\circ(\boldsymbol{d}-\boldsymbol{c})\big).
    \end{equation}
    And because\((-1)^{n+\epsilon_1+\dots+\epsilon_{i-1}+0+\epsilon_{i+1}+\dots+\epsilon_n}=(-1)\cdot(-1)^{n+\epsilon_1+\dots+\epsilon_{i-1}+1+\epsilon_{i+1}+\dots+\epsilon_n}\), 
    therefore from (12) and (13), (11) is proved.
    
    From (10), \(\Delta_{\widetilde{f}}^n[\boldsymbol{a},\boldsymbol{x}]\) can be written as\(\sum\limits_{\epsilon_1,\dots,\epsilon_n\in\{0,1\}}(-1)^{n+\epsilon_1+\dots+\epsilon_n}
            \Delta_f^n[\boldsymbol{a},\boldsymbol{x}]\).
    If there exists \(\epsilon_i=0 (i\neq k)\), from (11) we know
    \[\Delta_f^n[\boldsymbol{a},\boldsymbol{a}+\boldsymbol{\epsilon}\circ(\boldsymbol{x}-\boldsymbol{a})]=0.\]
    Therefore, \(\Delta_{\widetilde{f}}^n[\boldsymbol{a},\boldsymbol{x}]\) can be written as
        \begin{equation}
        \begin{aligned}[b]
            \Delta_{\widetilde{f}}^n[\boldsymbol{a},\boldsymbol{x}]&=
            \sum_{\epsilon_1,\dots,\epsilon_n\in\{0,1\}}(-1)^{n+\epsilon_1+\dots+\epsilon_n}
            \widetilde{f}\big(\boldsymbol{a}+\boldsymbol{\epsilon}\circ(\boldsymbol{x}-\boldsymbol{a})\big)\\
            &=\sum_{\epsilon_1,\dots,\epsilon_n\in\{0,1\}}(-1)^{n+\epsilon_1+\dots+\epsilon_n}
            \Delta_f^n[\boldsymbol{a},\boldsymbol{a}+\boldsymbol{\epsilon}\circ(\boldsymbol{x}-\boldsymbol{a})]\\
            &=0+(-1)^{n+1+\dots+1}\Delta_f^n[\boldsymbol{a},\boldsymbol{x}]\\
            &=\Delta_f^n[\boldsymbol{a},\boldsymbol{x}].
        \end{aligned}
    \end{equation}
    
    Because \(f\) is a jointly monotone increasing function on \([\boldsymbol{a},
    \boldsymbol{b}]\), from (14), we see that \(\widetilde{f}\) is a jointly monotone increasing function on \([\boldsymbol{a},
    \boldsymbol{b}]\).
    
    Next we show that \(\widetilde{f}\) is componentwise monotone increasing on \([\boldsymbol{a},
    \boldsymbol{b}]\). 
    
    For any \(a_k\leq c\leq d\leq b_k\), from (10), we can simplify the following expression to:
    \begin{equation}
        \begin{aligned}[b]
            &\Delta_{\widetilde{f}}^n[a_1,x_1]\times\dots\times[a_{k-1},x_{k-1}]\times[c,d]\times[a_{k+1},x_{k+1}]\times\dots\times[a_n,x_n]\\
            &=\sum_{\epsilon_1,\dots,\epsilon_n\in\{0,1\}}(-1)^{n+\epsilon_1+\dots+\epsilon_n}\widetilde{f}\big(\boldsymbol{a}|_{a_i=c}+\boldsymbol{\epsilon}\circ(\boldsymbol{x}|_{x_i=d}-\boldsymbol{a}|_{a_i=c})\big)\\
            &=\sum_{\epsilon_1,\dots,\epsilon_n\in\{0,1\}}(-1)^{n+\epsilon_1+\dots+\epsilon_n}\Delta_f^n[\boldsymbol{a}|_{a_i=c},\boldsymbol{a}|_{a_i=c}+\boldsymbol{\epsilon}\circ(\boldsymbol{x}|_{x_i=d}-\boldsymbol{a}|_{a_i=c})].
        \end{aligned}
    \end{equation}
    If there exists \(\epsilon_i=0(i\neq k)\), from (11) we know
    \begin{equation}
        \begin{aligned}[b]
            &\Delta_f^n[\boldsymbol{a}|_{a_i=c},\boldsymbol{a}|_{a_i=c}+\boldsymbol{\epsilon}\circ(\boldsymbol{x}|_{x_i=d}-\boldsymbol{a}|_{a_i=c})]\\
            &=\Delta_f^n[a_1,a_1+\epsilon_1(x_1-a_1)]\times\dots\times[a_i,a_i]\times\dots\times[a_n,a_n+\epsilon_n(x_n-a_n)]\\
            &=0.
        \end{aligned}
    \end{equation}
    Therefore from (16), can simplify (15) to:
    \begin{align*}
        &\Delta_{\widetilde{f}}^n[a_1,x_1]\times\dots\times[a_{k-1},x_{k-1}]\times[c,d]\times[a_{k+1},x_{k+1}]\times\dots\times[a_n,x_n]\\
        &=\sum_{\epsilon_1,\dots,\epsilon_n\in\{0,1\}}(-1)^{n+\epsilon_1+\dots+\epsilon_n}\Delta_f^n[\boldsymbol{a}|_{a_i=c},\boldsymbol{a}|_{a_i=c}+\boldsymbol{\epsilon}\circ(\boldsymbol{x}|_{x_i=d}-\boldsymbol{a}|_{a_i=c})]\\
        &=0+\sum_{\epsilon_k\in\{0,1\}}(-1)^{n+1+\dots+1+\epsilon_k+1+\dots+1}\Delta_f^n[\boldsymbol{a}|_{a_i=c},\big(x_1-a_1,\dots,c+\epsilon_k\cdot(d-c),\dots,x_n-a_n\big)]\\
        &=\widetilde{f}(x_1,\dots,x_{k-1},d,x_{k+1},\dots,x_n)-\widetilde{f}(x_1,\dots,x_{k-1},c,x_{k+1},\dots,x_n).
    \end{align*}
    Furthermore, because \(\widetilde{f}\) is a jointly monotone function, so
    \begin{equation}
        \begin{aligned}[b]
            &\Delta_{\widetilde{f}}^n[a_1,x_1]\times\dots\times[a_{k-1},x_{k-1}]\times[c,d]\times[a_{k+1},x_{k+1}]\times\dots\times[a_n,x_n]\\
            &=\widetilde{f}(x_1,\dots,x_{k-1},d,x_{k+1},\dots,x_n)-\widetilde{f}(x_1,\dots,x_{k-1},c,x_{k+1},\dots,x_n)\geq0.
        \end{aligned}      
    \end{equation}

    According to (17) and the definition of componentwise monotone function, prove if \(f(\boldsymbol{x})\) is a jointly monotone increasing function on \([\boldsymbol{a},\boldsymbol{b}]\), then \(\widetilde{f}\) is a componentwise monotone increasing function on \([\boldsymbol{a},\boldsymbol{b}]\).
\end{proof}

In the above proof, it is shown that \(\widetilde{f}\) is monotone increasing in each variable. Similarly, it can be shown that \(\widetilde{f}\) is jointly monotone increasing with respect to part of the variables. 

\begin{lemma}{(componentwise monotone functions are Lebesgue measurable, cf. [2])} 
    Let \(f(\boldsymbol{x})\) be a componentwise monotone function on \([\boldsymbol{a},\boldsymbol{b}]\). Then \(f\) is Lebesgue measurable on \([\boldsymbol{a},\boldsymbol{b}]\).\\[-0.5em]
\end{lemma}

\begin{proof}
    We only give the proof for the case \(\boldsymbol{a}\leq\boldsymbol{b}\) and \(f\) componentwise monotone increasing. In other cases, that is, when there exists \(a_i>b_i\) or \(f\) componentwise monotone decreasing, the conclusion can be similarly proved.

    Since one-dimensional monotone functions being measurable, one-dimensional componentwise monotone functions are measurable. If it has been proved that ($n-1$)-dimensional componentwise monotone functions are measurable on \([\boldsymbol{a},\boldsymbol{b}]\), next we prove that \(f\) is measurable on \([\boldsymbol{a},\boldsymbol{b}]\).

    Because \(a_i\leq b_i\), when \(\lambda\leq f(\boldsymbol{a})\), \(\{\boldsymbol{x}\mid \boldsymbol{x}\in[\boldsymbol{a},\boldsymbol{b}],f(\boldsymbol{x})<\lambda\}=\varnothing\). When \(\lambda>f(\boldsymbol{b})\), \(\{\boldsymbol{x}\mid \boldsymbol{x}\in[\boldsymbol{a},\boldsymbol{b}],f(\boldsymbol{x})<\lambda\}=[\boldsymbol{a},\boldsymbol{b}]\).

    Consider \(f(\boldsymbol{a})<\lambda\leq f(\boldsymbol{b})\), let \((x_1, \dots, x_{n-1}), (a_1, \dots, a_{n-1}), (b_1, \dots, b_{n-1})\) be \(\boldsymbol{x}', \boldsymbol{a}', \boldsymbol{b}'\). Let
    \[g(\boldsymbol{x}')=\sup\{x_n\mid \boldsymbol{x}'\in[\boldsymbol{a}',\boldsymbol{b}'],f(\boldsymbol{x})<\lambda\}.\]

    For \(\forall \boldsymbol{x_1}'<\boldsymbol{x_2}'\), because \(f(\boldsymbol{x})\) is a componentwise monotone increasing function, then there is
    \[g(\boldsymbol{x_1}')\geq g(\boldsymbol{x_2}').\]
    Therefore \(g\) is a componentwise monotone decreasing function on \([\boldsymbol{a}',\boldsymbol{b}']\subset\mathbb{R}^{n-1}\), then \(g\) is a measurable function. According to \([9,\rm{p}.204,Theorem 4.33]\),
    \[G(g)=\{\boldsymbol{x}\mid \boldsymbol{x}'\in[\boldsymbol{a}',\boldsymbol{b}'],0\leq x_n\leq g(\boldsymbol{x'})\}\]
    is a measurable set in \(\mathbb{R}^n\).

    Because
    \begin{equation}
        \{\boldsymbol{x}\mid \boldsymbol{x}'\in[\boldsymbol{a}',\boldsymbol{b}'],0\leq x_n< g(\boldsymbol{x}')\}\subset\{\boldsymbol{x}\mid \boldsymbol{x}\in[\boldsymbol{a},\boldsymbol{b}],f(\boldsymbol{x})<\lambda\},
    \end{equation}
    \begin{equation}
        \{\boldsymbol{x}\mid \boldsymbol{x}\in[\boldsymbol{a},\boldsymbol{b}],f(\boldsymbol{x})<\lambda\}\subset\{\boldsymbol{x}\mid \boldsymbol{x}'\in[\boldsymbol{a}',\boldsymbol{b}'],0\leq x_n\leq g(\boldsymbol{x}')\}.
    \end{equation}
    According to \([9,\rm{p}.204,Corollary4.32]\), we have
    \begin{equation}
        m\{\boldsymbol{x}\mid \boldsymbol{x}'\in[\boldsymbol{a}',\boldsymbol{b}'],x_n=g(\boldsymbol{x}')\}=0.
    \end{equation}
    Combining (18),(19),(20), we get that when \(f(\boldsymbol{a})< \lambda\leq f(\boldsymbol{b})\),
    \[\{\boldsymbol{x}\mid \boldsymbol{x}\in[\boldsymbol{a},\boldsymbol{b}],f(\boldsymbol{x})<\lambda\}\]
    is a measurable set on \([\boldsymbol{a},\boldsymbol{b}]\) and
    \[m\{\boldsymbol{x}\mid \boldsymbol{x}\in[\boldsymbol{a},\boldsymbol{b}],f(\boldsymbol{x})<\lambda\}=mG(g).\]

    By the definition of measurable function, for any \(\lambda\in\mathbb{R}\), \(\{\boldsymbol{x}\mid \boldsymbol{x}\in[\boldsymbol{a},\boldsymbol{b}],f(\boldsymbol{x})<\lambda\}\) is a measurable set, hence \(f\) is proved to be measurable on \([\boldsymbol{a},\boldsymbol{b}]\). 
\end{proof}
Now we prove Theorem 1.1.\\[-0.5em]
    
\begin{proof}
    Let \(f(\boldsymbol{x})\) be defined in some neighborhood of \(\boldsymbol{x}\in(\boldsymbol{a},\boldsymbol{b})\), let
    \[D^{\boldsymbol{1}}f(\boldsymbol{x})=\varlimsup_{h\to0^{+}}\frac{\Delta_f^n[\boldsymbol{x},\boldsymbol{x}+h]}{h^n},\]
    \[D_{\boldsymbol{1}}f(\boldsymbol{x})=\varliminf_{h\to0^{+}}\frac{\Delta_f^n[\boldsymbol{x},\boldsymbol{x}+h]}{h^n}.\]
    Because \(f\) is a jointly monotone increasing function, obviously
    \[D^{\boldsymbol{1}}f(\boldsymbol{x})\geq D_{\boldsymbol{1}}f(\boldsymbol{x})\geq0.\]
    \(f\) is differentiable at \(\boldsymbol{x}\), if and only if \(D^{\boldsymbol{1}}f(\boldsymbol{x})=D_{\boldsymbol{1}}f(\boldsymbol{x})\) is a finite value. 

    What needs to be proved is, except a zero measure set of \([\boldsymbol{a},\boldsymbol{b}]\), there is
    \begin{equation}
        D^{\boldsymbol{1}}f(\boldsymbol{x})=D_{\boldsymbol{1}}f(\boldsymbol{x}).
    \end{equation}
    This only needs to prove \(m(E)=0\), where
    \[E=\left\{\boldsymbol{x}\mid \boldsymbol{x}\in(\boldsymbol{a},\boldsymbol{b}),D^{\boldsymbol{1}}f(\boldsymbol{x})>D_{\boldsymbol{1}}f(\boldsymbol{x})\right\}.\]

    Since \(f\) is jointly monotone increasing, \(D^{\boldsymbol{1}}f(\boldsymbol{x})\) and \(D_{\boldsymbol{1}}f(\boldsymbol{x})\) are nonnegative. Decompose
    \[A=\bigcup_{r,s\in\mathbb{Q}^{+}}\left\{\boldsymbol{x}\mid \boldsymbol{x}\in(\boldsymbol{a},\boldsymbol{b}),D^{\boldsymbol{1}}f(\boldsymbol{x})>r>s>D_{\boldsymbol{1}}f(\boldsymbol{x})\right\},\]
    where \(\mathbb{Q}^+\) is the set of all positive rational numbers.
    Let
    \[A_{rs}=\left\{\boldsymbol{x}\mid \boldsymbol{x}\in(\boldsymbol{a},\boldsymbol{b}),D^{\boldsymbol{1}}f(\boldsymbol{x})>r>s>D_{\boldsymbol{1}}f(\boldsymbol{x})\right\}.\]
    Then, if we prove that \(\forall r,s\in\mathbb{Q}^{+}\), we have \(m(A_{rs})=0\), it follows from the countability of \(\boldsymbol{Q}^{+}\) that \(m(A)=m(E)=0\).
    
    Proof by contradiction. Suppose not, i.e., \(m^{*}(A_{rs})>0\). Construct an open set \(G\) such that \(A_{rs}\subset G\subset[\boldsymbol{0},\boldsymbol{a}]\), and \(m(G)<(1+\varepsilon)m^{*}(A_{rs})\). \(\forall \boldsymbol{x}\in A_{rs}\), from \(D_{\boldsymbol{1}}f(\boldsymbol{x})<s\), we can take a sufficiently small positive number \(h\) such that
    \begin{equation}
        \frac{\Delta_f^n[\boldsymbol{x},\boldsymbol{x}+h]}{h^n}<s.
    \end{equation}
    Assume without loss of generality that \([\boldsymbol{x},\boldsymbol{x}+h]\subset G\). Thus, the collection of all such \([\boldsymbol{x},\boldsymbol{x}+h]\) forms a Vitali covering of \(A_{rs}\). Therefore, according to the Vitali covering theorem, there exist closed rectangles \([\boldsymbol{x_1},\boldsymbol{x_1}+h_1],\dots,[\boldsymbol{x_N},\boldsymbol{x_N}+h_N]\) with disjoint interiors such that
    \[m^{*}\left(A_{rs}\backslash\bigcup_{i=1}^{N}[\boldsymbol{x_i},\boldsymbol{x_i}+h_i]\right)<\varepsilon,\]
    hence 
    \[m^{*}\left(A_{rs}\cap\bigcup_{i=1}^{N}[\boldsymbol{x_i},\boldsymbol{x_i}+h_i]\right)>m^{*}(A_{rs})-\varepsilon,\]
    and 
    \[\sum_{i=1}^{N}h^n\leqslant m(G)<(1+\varepsilon)m^{*}(A_{rs}).\]
    From (22), 
    \[\Delta_f^n[\boldsymbol{x_i},\boldsymbol{x_i}+h_i]<sh^n(i=1,2,\dots,N).\]
    Therefore, 
    \begin{equation}
        \sum_{i=1}^{N}\Delta_f^n[\boldsymbol{x_i},\boldsymbol{x_i}+h_i]<s\sum_{i=1}^{N}h^n<s(1+\varepsilon)m^{*}(A_{rs}).
    \end{equation}
    
    Let
    \[B_{rs}=A_{rs}\cap\bigcup_{i=1}^{N}(\boldsymbol{x_i},\boldsymbol{x_i}+h_i).\]
    Similar to before, for any \(\boldsymbol{x}\in B_{rs}\), we have \(D^{\boldsymbol{1}}f(\boldsymbol{x})>r\). Therefore, we can take a sufficiently small positive number \(k\) such that \([\boldsymbol{x},\boldsymbol{x}+k]\) is contained within some \([\boldsymbol{x_i},\boldsymbol{x_i}+h_i]\), and
    \begin{equation}
        \frac{\Delta_f^n[\boldsymbol{x},\boldsymbol{x}+k]}{k^n}>r.
    \end{equation}
    The collection of all such \([\boldsymbol{x},\boldsymbol{x}+k]\) forms a Vitali covering of \(B_{rs}\). Again by the Vitali covering theorem, there exists a sequence of closed rectangles \([\boldsymbol{x_j},\boldsymbol{x_j}+k_j](j=1,\dots,m)\) with disjoint interiors such that
    \begin{equation}
        \sum_{j=1}^{m}k_{j}^n>m^{*}(B_{rs})-\varepsilon,
    \end{equation}
    and from (24), 
    \[\Delta_f^n[\boldsymbol{x_j},\boldsymbol{x_j}+k_j]>rk_j^n(j=1,2,\dots,m).\]
    Thus, from (24) and (25), we know 
    \begin{equation}
        \sum_{j=1}^{m}\Delta_f^n[\boldsymbol{x_j},\boldsymbol{x_j}+k_j]
        >r\sum_{j=1}^{m}k_j^n\\
        >r(m^{*}(B_{rs})-\varepsilon)>r(m^{*}(A_{rs})-2\varepsilon).
    \end{equation}
    
    Note that \(f(\boldsymbol{x})\) is jointly monotone increasing on \([\boldsymbol{a},\boldsymbol{b}]\), and \([\boldsymbol{x_j},\boldsymbol{x_j}+k_j]\) is contained within some \([\boldsymbol{x_i},\boldsymbol{x_i}+h_i]\), therefore
    \[\sum_{j=1}^{m}\Delta_f^n[\boldsymbol{x_j},\boldsymbol{x_j}+k_j]\leqslant\sum_{i=1}^{N}\Delta_f^n[\boldsymbol{x_i},\boldsymbol{x_i}+h_i].\]
    Combining (23) and (26), we get 
    \[r(m^{*}(A_{rs})-2\varepsilon)<s(1+\varepsilon)m^{*}(A_{rs}).\]
    By the arbitrariness of \(\varepsilon\), we get \(rm^*(A_{rs})<sm^*(A_{rs})\), hence \(m^*(A)=0\), a contradiction. This proves that \(f\) has a nonnegative derivative \(f^{(n)}(\boldsymbol{x})\)(possibly\(+\infty)\) almost everywhere on \([\boldsymbol{a},\boldsymbol{b}]\).

    We now prove that \(f\) is a measurable function on \([\boldsymbol{a},\boldsymbol{b}]\).

    Let
    \begin{equation}
        \begin{aligned}[b]
            \widetilde{f_N}(\boldsymbol{x})&=N^{n}\Delta_{\widetilde{f}}^n[\boldsymbol{x},\boldsymbol{x}+\frac{1}{N}]\\
            &=N^{n}\sum_{\epsilon_1,\dots,\epsilon_n\in\{0,1\}}(-1)^{n+\epsilon_1+\dots+\epsilon_n}\widetilde{f}(\boldsymbol{x}+\frac{1}{N}\boldsymbol{\epsilon}),
            \boldsymbol{x}\in [\boldsymbol{a},\boldsymbol{b}].
        \end{aligned}
    \end{equation}
    Since \(f\) has \(f^{(n)}\) almost everywhere, then
    \begin{equation}
        \lim_{N\to+\infty}\widetilde{f_N}(\boldsymbol{x})=f^{(n)}(\boldsymbol{x}).
    \end{equation}
    Extend the definition of \(f\) to \([\boldsymbol{a},\boldsymbol{b}+1]\) by defining
    \[f(\boldsymbol{x})=f\big(\boldsymbol{x}+\boldsymbol{1}_{\boldsymbol{x}>\boldsymbol{b}}\circ(\boldsymbol{b}-\boldsymbol{x})\big).\]
    By Lemma 2.1, it is easy to prove that \(f\) is a jointly monotone increasing function on \([\boldsymbol{a},\boldsymbol{b}+1]\), so \(\widetilde{f}\) is a componentwise monotone increasing function on \([\boldsymbol{a},\boldsymbol{b}+1]\). Therefore, by Lemma 2.2, \(\widetilde{f}\) is a measurable function on \([\boldsymbol{a},\boldsymbol{b}+1]\). Because\(\widetilde{f}\) is measurable on \([\boldsymbol{a},\boldsymbol{b}+1]\), from (27), \(\widetilde{f_N}\) is measurable on \([\boldsymbol{a},\boldsymbol{b}]\). Since the limit of a sequence of measurable functions is measurable, and from (28), it is proved that \(f\) is measurable on \([\boldsymbol{a},\boldsymbol{b}]\).

    Measurability implies integrability. Next, prove
    \begin{equation}
        \int_{\boldsymbol{a}}^{\boldsymbol{b}}f^{(n)}(\boldsymbol{x})\diff \boldsymbol{x}\leq 2^{n-1}\Delta_f^n[\boldsymbol{a},\boldsymbol{b}].   
    \end{equation}
    Because \(\widetilde{f}\) is a componentwise monotone increasing function on \([\boldsymbol{a},\boldsymbol{b}+1]\), i.e., for any \(\boldsymbol{c},\boldsymbol{d}\in[\boldsymbol{a},\boldsymbol{b}+1],\boldsymbol{c}\leq\boldsymbol{d}\), we have
    \begin{equation}
        \widetilde{f}(\boldsymbol{c})\leq\widetilde{f}(\boldsymbol{d}).
    \end{equation}
    By Fatou's lemma,
    \begin{equation}
        \begin{aligned}[b]
            \int_{\boldsymbol{a}}^{\boldsymbol{b}}f^{(n)}(\boldsymbol{x})\diff \boldsymbol{x}&=\int_{\boldsymbol{a}}^{\boldsymbol{b}}\varliminf_{N\to+\infty}\widetilde{f_N}(\boldsymbol{x})\diff \boldsymbol{x}\\
            &\leq\varliminf_{N\to+\infty}\int_{\boldsymbol{a}}^{\boldsymbol{b}}
            \widetilde{f_N}(\boldsymbol{x})\diff \boldsymbol{x}\\
            &=\varliminf_{N\to+\infty}N^{n}\int_{\boldsymbol{a}}^{\boldsymbol{b}}
            \Delta_{\widetilde{f}}^n[\boldsymbol{x},\boldsymbol{x}+\frac{1}{N}]\diff \boldsymbol{x}\\
            &=\varliminf_{N\to+\infty}N^n\int_{\boldsymbol{a}}^{\boldsymbol{b}}
            \sum_{\epsilon_1,\dots,\epsilon_n\in\{0,1\}}(-1)^{n+\epsilon_1+\dots+\epsilon_n}\widetilde{f}(\boldsymbol{x}+\frac{1}{N}\boldsymbol{\epsilon})\diff \boldsymbol{x}\\
            &=\varliminf_{N\to+\infty}N^n\left(\sum_{\epsilon_1,\dots,\epsilon_n\in\{0,1\}}(-1)^{n+\epsilon_1+\dots+\epsilon_n} \int_{\boldsymbol{a}}^{\boldsymbol{b}}\widetilde{f}(\boldsymbol{x}+\frac{1}{N}\boldsymbol{\epsilon})\diff \boldsymbol{x}\right)\\
            &=\varliminf_{N\to+\infty}N^n\left(\sum_{\epsilon_1,\dots,\epsilon_n\in\{0,1\}}(-1)^{n+\epsilon_1+\dots+\epsilon_n}
            \int_{\boldsymbol{a}+\frac{1}{N}\boldsymbol{\epsilon}}^{\boldsymbol{b}+\frac{1}{N}\boldsymbol{\epsilon}}\widetilde{f}(\boldsymbol{x})\diff \boldsymbol{x}\right)\\
            &=\varliminf_{N\to+\infty}N^n\left(\sum_{\epsilon_1,\dots,\epsilon_n\in\{0,1\}}(-1)^{n+\epsilon_1+\dots+\epsilon_n}
            \int_{\boldsymbol{a}+\boldsymbol{\epsilon}\circ(\boldsymbol{b}-\boldsymbol{a})}^{\boldsymbol{a}+\boldsymbol{\epsilon}\circ(\boldsymbol{b}-\boldsymbol{a})+\frac{1}{N}}\widetilde{f}(\boldsymbol{x})\diff \boldsymbol{x}\right).\\
        \end{aligned}
    \end{equation}
    \begin{figure}[h]
        \centering
        \subfloat[]{\includegraphics[width=0.4\textwidth]{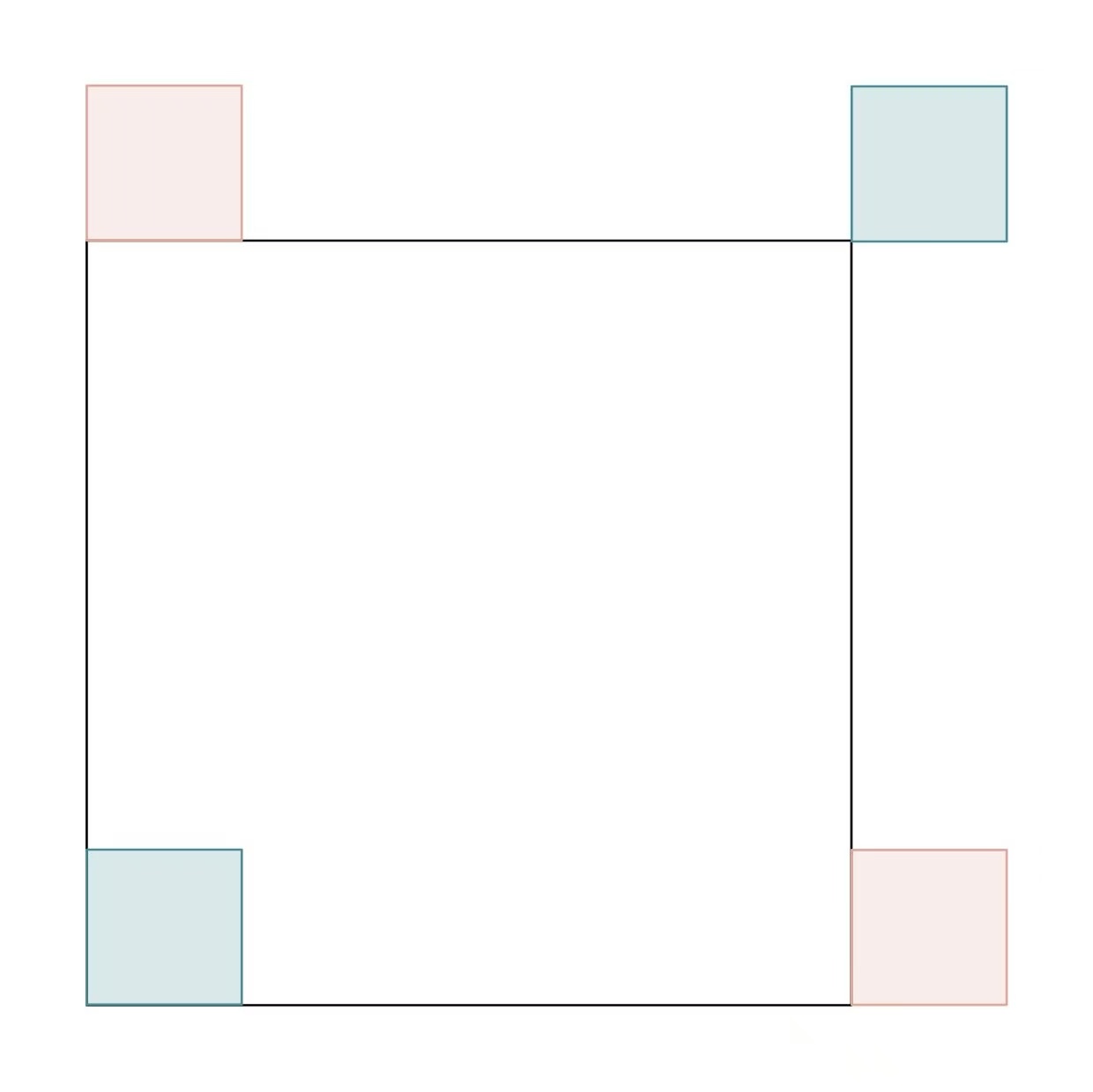}\label{fig:sub1}}
        \hfill
        \subfloat[]{\includegraphics[width=0.4\textwidth]{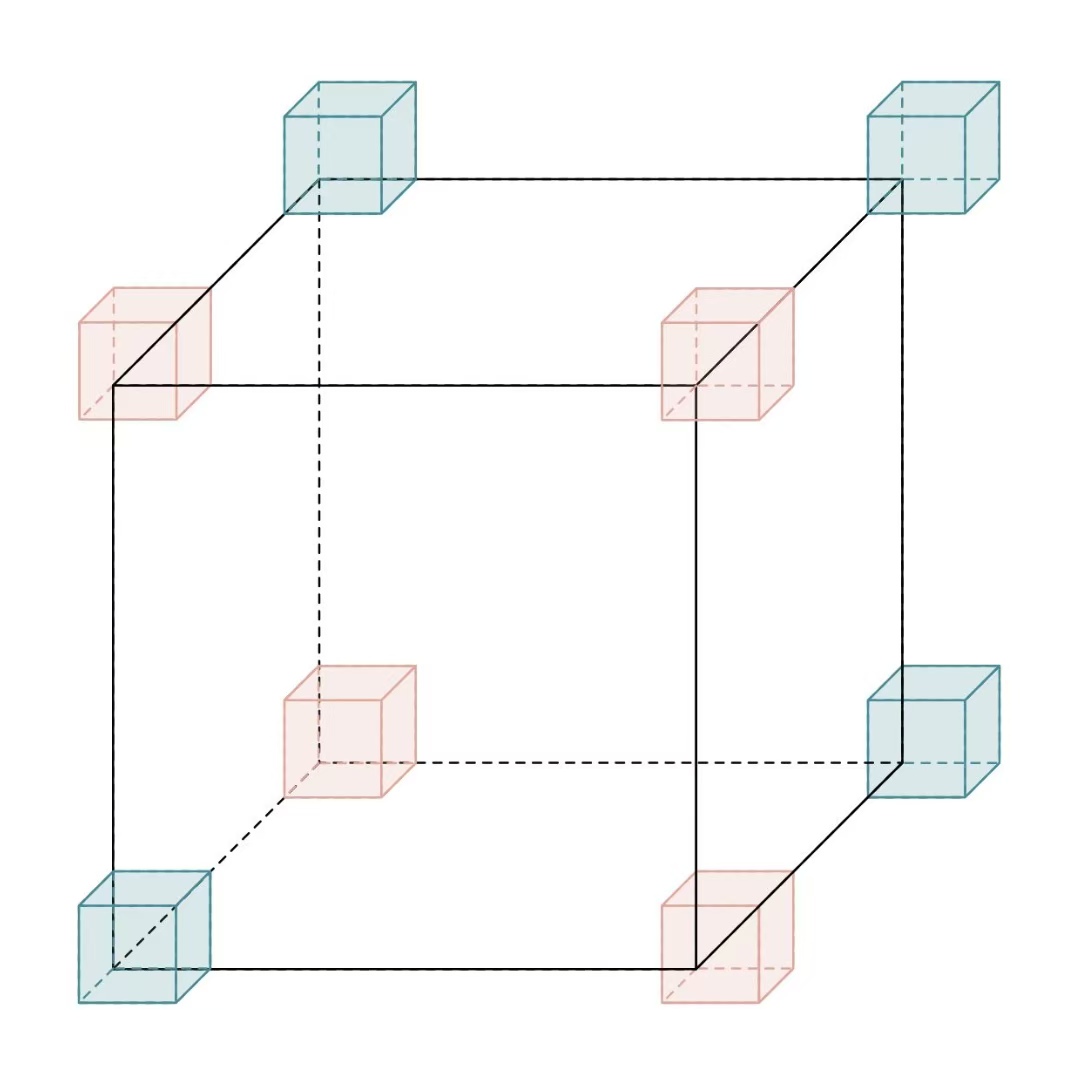}\label{fig:sub2}}
        \caption{(a) An illustration of the two-dimensional case. The side lengths of the four rectangles are \(\frac{1}{N}\); (31) represents the sum of the integrals of \(\widetilde{f}\) over the two blue rectangles minus the sum of the integrals of \(\widetilde{f}\) over the two red rectangles.\linebreak (b) An illustration of the three-dimensional case. The side lengths of the eight rectangles are \(\frac{1}{N}\); (31) represents the sum of the integrals of \(\widetilde{f}\) over the four blue rectangles minus the sum of the integrals of \(\widetilde{f}\) over the four red rectangles.}
        \label{fig:total}
    \end{figure} 
    
    From (30), (31) can be simplified to
    \begin{equation}
        \begin{aligned}[b]
            \int_{\boldsymbol{a}}^{\boldsymbol{b}}f^{(n)}(\boldsymbol{x})\diff \boldsymbol{x}&\leq\varliminf_{N\to+\infty}N^n\left(\sum_{\epsilon_1,\dots,\epsilon_n\in\{0,1\}}(-1)^{n+\epsilon_1+\dots+\epsilon_n}
            \int_{\boldsymbol{\boldsymbol{a}+\epsilon}\circ(\boldsymbol{b}-\boldsymbol{a})}^{\boldsymbol{a}+\boldsymbol{\epsilon}\circ(\boldsymbol{b}-\boldsymbol{a})+\frac{1}{N}}\widetilde{f}(\boldsymbol{x})\diff \boldsymbol{x}\right)\\
        &\leq\varliminf_{N\to+\infty}N^n\left(\frac{1}{N^n}\sum_{n+\epsilon_1+\dots+\epsilon_n\equiv0\pmod2}\widetilde{f}\left(\boldsymbol{a}+\frac{1}{N}+\boldsymbol{\epsilon}\circ(\boldsymbol{b}-\boldsymbol{a}-\frac{1}{N})\right)\right.\\
        &-\left.\frac{1}{N^n}\sum_{n+\epsilon_1+\dots+\epsilon_n\equiv1\pmod2}\widetilde{f}\big(\boldsymbol{a}+\boldsymbol{\epsilon}\circ(\boldsymbol{b}-\boldsymbol{a})\big)\right)\\
        \end{aligned}
    \end{equation}
    
    Because in the term \(\sum\limits_{n+\epsilon_1+\dots+\epsilon_n\equiv1\pmod2}\widetilde{f}\big(\boldsymbol{a}+\boldsymbol{\epsilon}\circ(\boldsymbol{b}-\boldsymbol{a})\big)\) there exists at least one \(\epsilon_i=0\), from  (10) and (11) we know this term is 0. That is, (32) can be simplified to
    \begin{align*}
        \int_{\boldsymbol{a}}^{\boldsymbol{b}}f^{(n)}(\boldsymbol{x})\diff \boldsymbol{x}&\leq\varliminf_{N\to+\infty}N^n\left(\frac{1}{N^n}\sum_{n+\epsilon_1+\dots+\epsilon_n\equiv0\pmod2}\widetilde{f}\left(\boldsymbol{a}+\frac{1}{N}+\boldsymbol{\epsilon}\circ(\boldsymbol{b}-\boldsymbol{a}-\frac{1}{N})\right)-0\right)\\
        &=\varliminf_{N\to+\infty}\sum_{n+\epsilon_1+\dots+\epsilon_n\equiv0\pmod2}\widetilde{f}\left(\boldsymbol{a}+\frac{1}{N}+\boldsymbol{\epsilon}\circ(\boldsymbol{b}-\boldsymbol{a}-\frac{1}{N})\right)\\
        &=\varliminf_{N\to+\infty}\sum_{n+\epsilon_1+\dots+\epsilon_n\equiv0\pmod2}\Delta_f^n[\boldsymbol{a},\boldsymbol{a}+\frac{1}{N}+\boldsymbol{\epsilon}\circ(\boldsymbol{b}-\boldsymbol{a}-\frac{1}{N})]\\
        &\leq 2^{n-1}\Delta_f^n[\boldsymbol{a},\boldsymbol{b}].
    \end{align*}
    Thus, (29) is proved.
   
    Now prove equation (4). Extend \(f\) from \([\boldsymbol{a},\boldsymbol{b}]\) to \([\boldsymbol{a}-1,\boldsymbol{b}]\) by defining
    \[f(\boldsymbol{x})=f\big(\boldsymbol{a}+\boldsymbol{1}_{\boldsymbol{}\boldsymbol{x}>\boldsymbol{a}}\circ(\boldsymbol{x}-\boldsymbol{a})\big),\boldsymbol{x}\in[\boldsymbol{a}-1,\boldsymbol{b}].\]
    From (29), we know \(f^{(n)}(\boldsymbol{x})\) is integrable on \([\boldsymbol{a},\boldsymbol{b}]\). It is easy to prove that \(f^{(n)}\) is identically 0 on \([\boldsymbol{a}-1,\boldsymbol{b}]\setminus [\boldsymbol{a},\boldsymbol{b}]\), so \(f^{(n)}\) is also integrable on \([\boldsymbol{a}-1,\boldsymbol{b}]\), and its integral equals the integral on\([\boldsymbol{a},\boldsymbol{b}]\). That is,
    \begin{align*}
        \int_{\boldsymbol{a}-1}^{\boldsymbol{b}}f^{(n)}(\boldsymbol{x})\diff \boldsymbol{x}&=\int_{\boldsymbol{a}}^{\boldsymbol{b}}f^{(n)}(\boldsymbol{x})\diff \boldsymbol{x}\\
        &\leq \varliminf_{N\to+\infty}\sum_{n+\epsilon_1+\dots+\epsilon_n\equiv0\pmod2}\Delta_f^n[\boldsymbol{a}-1,\boldsymbol{a}-1+\frac{1}{N}+\boldsymbol{\epsilon}\circ(\boldsymbol{b}-\boldsymbol{a}+1-\frac{1}{N})]\\
        &=0+\varliminf_{N\to+\infty}\Delta_f^n[\boldsymbol{a}-1,\boldsymbol{a}-1+\frac{1}{N}+\boldsymbol{1}\circ(\boldsymbol{b}-\boldsymbol{a}+1-\frac{1}{N})]\\
        &=\Delta_f^n[\boldsymbol{a}-1,\boldsymbol{b}]\\
        &=\Delta_f^n[\boldsymbol{a},\boldsymbol{b}].
    \end{align*}
    Thus, it is proved that
    \[\int_{\boldsymbol{a}}^{\boldsymbol{b}}f^{(n)}(\boldsymbol{x})\diff \boldsymbol{x}\leq \Delta_f^n[\boldsymbol{a},\boldsymbol{b}].\]
\end{proof}

\subsection{Proof of Proposition 1.2}
\begin{proof}
    First, we show that the Dini derivatives exist almost everywhere in each octant.
    
    The existence almost everywhere of the Dini derivative in the \(\boldsymbol{1}=(1,\dots,1)\) octant, i.e.,
    \[\varlimsup_{h\to0^+}\frac{\Delta_f^n[\boldsymbol{x,\boldsymbol{x}+h\boldsymbol{1}}]}{h^n}=\varliminf_{h\to0^+}\frac{\Delta_f^n[\boldsymbol{x,\boldsymbol{x}+h\boldsymbol{1}}]}{h^n}.\]
    has been given in the proof of Theorem 1.1 and can be proved by (21). The existence almost everywhere of the Dini derivatives in the other octants can be proved similarly.

    We now prove that the Dini derivatives in all octants are almost everywhere equal.

    We only give the proof for \(D_{1,1}f(x,y)=D_{-1,1}f(x,y)\) in the two-dimensional case, i.e., the Dini derivative in the first quadrant is almost everywhere equal to the Dini derivative in the second quadrant. The almost everywhere equality of the first quadrant Dini derivative with the third and fourth quadrant Dini derivatives in two dimensions can be proved similarly. The high-dimensional case can be proved similarly.

    Take the two-dimensional case as an example. We prove that for almost every \((x,y)\in[a_1,b_1]\times[a_2,b_2],D_{1,1}f(x,y)=D_{-1,1}f(x,y)\).
    \begin{align*}
        D_{1,1}f(x,y)=\lim_{h\to0^{+}}\frac{\Delta_f^2[x,x+h]\times[y,y+h]}{h^2},\\
        D_{-1,1}f(x,y)=\lim_{h\to0^{+}}\frac{\Delta_f^2[x-h,x]\times[y,y+h]}{h^2}.
    \end{align*}
    This only requires proving \(m(E_1)=m(E_2)=0\), where
    \[E_1=\left\{(x,y)\mid(x,y)\in(a_1,b_1)\times(a_2,b_2), D_{1,1}f(x,y)>D_{-1,1}f(x,y)\right\}.\]
    \[E_2=\left\{(x,y)\mid(x,y)\in(a_1,b_1)\times(a_2,b_2), D_{1,1}f(x,y)<D_{-1,1}f(x,y)\right\}.\]
    
    Decompose
    \[A=\bigcup_{r,s\in\mathbb{Q}^{+}}\left\{(x,y)\mid (x,y)\in(a_1,b_1)\times(a_2,b_2),D_{1,1}f(x,y)>r>s>D_{-1,1}f(x,y)\right\},\]
    where \(Q^+\) is the set of all positive rational numbers. Let
    \[A_{rs}=\left\{(x,y)\mid (x,y)\in(a_1,b_1)\times(a_2,b_2),D_{1,1}f(x,y)>r>s>D_{-1,1}f(x,y)\right\}.\]
    Then, if we prove that  \(\forall r,s\in\mathbb{Q}^{+}\), we have \(m(A_{rs})=0\), it follows from the countability of \(\mathbf{Q}^{+}\) that \(m(A)=m(E_1)=0\).
    
    Proof by contradiction. Suppose not, i.e., \(m^{*}(A_{rs})>0\). Construct an open set \(G\) such that \(A_{rs}\subset G\subset[\boldsymbol{a},\boldsymbol{b}]\), and \(m(G)<(1+\varepsilon)m^{*}(A_{rs})\).
    \(\forall (x,y)\in A_{rs}\), from \(D_{-1,1}f(x,y)<s\), we can take a sufficiently small positive number \(h\) such that
    \begin{equation}
        \frac{\Delta_f^2[x-h,x]\times[y,y+h]}{h^2}<s.
    \end{equation}
    Assume without loss of generality that \([x-h,x]\times[y,y+h]\subset G\). Thus, the collection of all such \([x-h,x]\times[y,y+h]\) forms a Vitali covering of \(A_{rs}\). Therefore, according to the Vitali covering theorem, there exist rectangles 
    \([x_1-h_1,x_1]\times[y_1,y_1+h_1],\dots,[x_N-h_N,x_N]\times[y_N,y_N+h_N]\) with disjoint interiors such that
    \[m^{*}\left(A_{rs}\backslash\bigcup_{i=1}^{N}[x_i-h_i,x_i]\times[y_i,y_i+h_i]\right)<\varepsilon,\]
    hence
    \[m^{*}\left(A_{rs}\cap\bigcup_{i=1}^{N}[x_i-h_i,x_i]\times[y_i,y_i+h_i]\right)>m^{*}(A_{rs})-\varepsilon,\]
    and
    \[\sum_{i=1}^{N}h_i^2\leqslant m(G)<(1+\varepsilon)m^{*}(A_{rs}).\]
    From (33),
    \[\Delta_f^2[x_i-h_i,x_i],\times[y_i,y_i+h_i]<sh_i^2(i=1,2,\dots,N).\]
    Therefore,
    \begin{equation}
        \sum_{i=1}^{N}\Delta_f^2[x_i-h_i,x_i]\times[y_i,y_i+h_i]<s\sum_{i=1}^{N}h_i^2<s(1+\varepsilon)m^{*}(A_{rs}).
    \end{equation}
    
    Let
    \[B_{rs}=A_{rs}\cap\bigcup_{i=1}^{N}(x_i-h_i,x_i)\times(y_i,y_i+h_i).\]
    Similar to before, for any \((x,y)\in B_{rs}\), we have \(D_{1,1}f(x,y)>r\). 
    Therefore, for any \((x,y)\in B_{rs}\), we can take a sufficiently small positive number \(k\) such that \([x,x+k]\times[y,y+k]\) is contained within some \([x_i-h_i,x_i]\times[y_i,y_i+h_i]\), and
    \begin{equation}
        \frac{\Delta_f^2[x,x+k]\times[y,y+k]}{k^2}>r.
    \end{equation}
    The collection of all such \([x,x+k]\times[y,y+k]\) forms a Vitali covering of \(B_{rs}\).

    Again by the Vitali covering theorem, there exists a sequence of rectangles 
    \([x_j,x_j+k_j]\times[y_j,y_j+k_j](j=1,\dots,m)\) with disjoint interiors such that 
    \begin{equation}
        \sum_{j=1}^{m}k_{j}^2>m^{*}(B_{rs})-\varepsilon,
    \end{equation}
    and from (35), 
    \[\Delta_f^2[x_j,x_j+k_j]\times[y_j,y_j+k_j]>rk_j^2(j=1,2,\dots,m).\]
    Thus, from (35) and (36), we know 
    \begin{equation}
        \sum_{j=1}^{m}\Delta_f^2[x_j,x_j+k_j]\times[y_j,y_j+k_j]
        >r\sum_{j=1}^{m}k_j^2\\
        >r(m^{*}(B_{rs})-\varepsilon)>r(m^{*}(A_{rs})-2\varepsilon).
    \end{equation}
    Note that \(f(x,y)\) is jointly monotone increasing on \([a_1,b_1]\times[a_2,b_2]\), and \([x_j,x_j+k_j]\times[y_j,y_j+k_j]\) is contained within some \([x_i-h_i,x_i]\times[y_i,y_i+h_i]\), therefore
    \[\sum_{j=1}^{m}\Delta_f^2[x_j,x_j+k_j]\times[y_j,y_j+k_j]\leqslant\sum_{i=1}^{N}\Delta_f^2[x_i-h_i,x_i]\times[y_i,y_i+h_i].\]
    Combining (34) and (37), we get
    \[r(m^{*}(A_{rs})-2\varepsilon)<s(1+\varepsilon)m^{*}(A_{rs}).\]
    By the arbitrariness of \(\varepsilon\), simplifying gives \(rm^*(A_{rs})<sm^*(A_{rs})\), a contradiction. Hence \(m^*(A_{rs})=0\), which proves \(m(A)=m(E_1)=0\).

    It can be similarly proved that \(m(E_2)=0\). From \(m(E_1)=m(E_2)=0\), it is proved that \(D_{1,1}f(x,y)\) equals \(D_{-1,1}f(x,y)\) almost everywhere on \([a_1,b_1]\times[a_2,b_2]\).
\end{proof}

\subsection{Proof of Corollary 1.3}
The necessity follows directly from Corollary 1.9. So we only prove the sufficiency, that is, to prove: If a jointly monotone increasing function \(f\) on \([\boldsymbol{a},\boldsymbol{b}]\) satisfies
\begin{equation}
    \int_{\boldsymbol{a}}^{\boldsymbol{b}}f^{(n)}(\boldsymbol{x})\diff\boldsymbol{x}=\Delta_f^n[\boldsymbol{a},\boldsymbol{b}],
\end{equation}
then \(f\in AC[\boldsymbol{a},\boldsymbol{b}]\).

\begin{proof}
    We only give the proof for the two-dimensional case. The high-dimensional case can be proved similarly.
    
    We prove that for any \((x,y)\in[a_1,b_1]\times[a_2,b_2]\), we have
    \begin{equation}
        \int_{a_1}^{x}\int_{a_2}^{y}f^{(2)}\diff x\diff y=\Delta_f^2[a_1,x]\times[a_2,y].
    \end{equation}
    We know
    \begin{equation}
        \begin{aligned}[b]
            &\int_{a_1}^{x}\int_{a_2}^{y}f^{(2)}\diff x\diff y\\
            &=
            \int_{a_1}^{b_1}\int_{a_2}^{b_2}f^{(2)}\diff x\diff y-
            \int_{a_1}^{x}\int_{y}^{b_2}f^{(2)}\diff x\diff y-
            \int_{x}^{b_1}\int_{a_2}^{y}f^{(2)}\diff x\diff y-
            \int_{x}^{b_1}\int_{y}^{b_2}f^{(2)}\diff x\diff y.
        \end{aligned}
    \end{equation}
    From (38) and Theorem 1.1, (40) can be bounded by:
    \begin{equation}
        \begin{aligned}[b]
            &\int_{a_1}^{x}\int_{a_2}^{y}f^{(2)}\diff x\diff y\\
            &\geq\Delta_f^2[a_1,b_1]\times[a_2,b_2]-\Delta_f^2[a_1,x]\times[y,b_2]-\Delta_f^2[x,b_1]\times[a_2,y]-\Delta_f^2[x,b_1]\times[y,b_2]\\
            &=\Delta_f^2[a_1,x]\times[a_2,y].
        \end{aligned}
    \end{equation}
    From Theorem 1.1, we know
    \begin{equation}
        \int_{a_1}^{x}\int_{a_2}^{y}f^{(2)}\diff x\diff y\leq\Delta_f^2[a_1,x]\times[a_2,y].
    \end{equation}
    Combining (41) and (42), (39) is proved.

    We now prove that for any \(x_1\leq x_2\in[a_1,a_2],y_1\leq y_2\in[b_1,b_2]\), we have
    \begin{equation}
        \int_{x_1}^{x_2}\int_{y_1}^{y_2}f^{(2)}\diff x\diff y=\Delta_f^2[x_1,x_2]\times[y_1,y_2].
    \end{equation}
    From (39), (43) can be proved:
    \begin{align*}
        &\int_{x_1}^{x_2}\int_{y_1}^{y_2}f^{(2)}\diff x\diff y\\
        &=\int_{a_1}^{x_2}\int_{a_2}^{y_2}f^{(2)}\diff x\diff y-\int_{a_1}^{x_1}\int_{a_2}^{y_2}f^{(2)}\diff x\diff y-
            \int_{a_1}^{x_2}\int_{a_2}^{y_1}f^{(2)}\diff x\diff y+
            \int_{a_1}^{x_1}\int_{a_2}^{y_1}f^{(2)}\diff x\diff y\\
        &=\Delta_f^2[a_1,x_2]\times[a_2,y_2]-\Delta_f^2[a_1,x_1]\times[a_2,y_2]-\Delta_f^2[a_1,x_2]\times[a_2,y_1]+\Delta_f^2[a_1,x_1]\times[a_2,y_1]\\
        &=\Delta_f^2[x_1,x_2]\times[y_1,y_2].
    \end{align*}    

    We now prove that \(f\) is absolutely continuous on the domain. From (43) we know 
    \[\Delta_f^2[x_1,x_2]\times[y_1,y_2]=\int_{x_1}^{x_2}\int_{y_1}^{y_2}f^{(2)}\diff x\diff y.\]
    Since \(f^{(2)}(x,y)\) is an integrable function, by the absolute continuity of integrable functions and the definition of absolute continuity, it can be proved that \(f\) is absolutely continuous on \([a_1,b_1]\times[a_2,b_2]\).
\end{proof}

\subsection{Proof of Theorem 1.4}
Before proving Theorem 1.4, we first give a lemma.
\begin{lemma} 
    If \(f\) is a real-valued function defined on \([\boldsymbol{a},\boldsymbol{b}]\in\mathbb{R}^n\), given \(c_i\in[a_i,b_i](i=1,\dots n)\), let \([a_i,c_i]\) be \(I_{i0}\), and \([c_i,b_i]\) be \(I_{i1}\), then
    \[I_{1j}\times I_{2j}\times\dots\times I_{nj},j\in\{0,1\}\]
    divides \([\boldsymbol{a},\boldsymbol{b}]\) into \(2^n\) n-dimensional rectangles, satisfying
    \begin{equation}
        \bigvee_{\boldsymbol{a}}^{\boldsymbol{b}}(f)=\sum_{j\in\{0,1\}}\bigvee_{I_{1j}\times\dots\times I_{nj}}(f).
    \end{equation}
\end{lemma}
\begin{remark}
    Lemma 2.3 shows that if a rectangle is divided into smaller rectangles, the sum of the total variations on the smaller rectangles equals the total variation on the original rectangle, indicating that the total variation is additive over domain.  
\end{remark}

\begin{proof}
    First prove
    \begin{equation}
        \bigvee_{\boldsymbol{a}}^{\boldsymbol{b}}(f)\leq\sum_{j\in\{0,1\}}\bigvee_{I_{1j}\times\dots\times I_{nj}}(f).
    \end{equation}
    Let the original partition be \(\Delta_1\). Let \(\Delta_2\) be the partition obtained by adding the point \(\boldsymbol{c}=(c_1,\dots,c_n)\) to \(\Delta_1\), see Fig. 2.2.
    \begin{figure}[h]
        \centering
        \includegraphics[width=8cm]{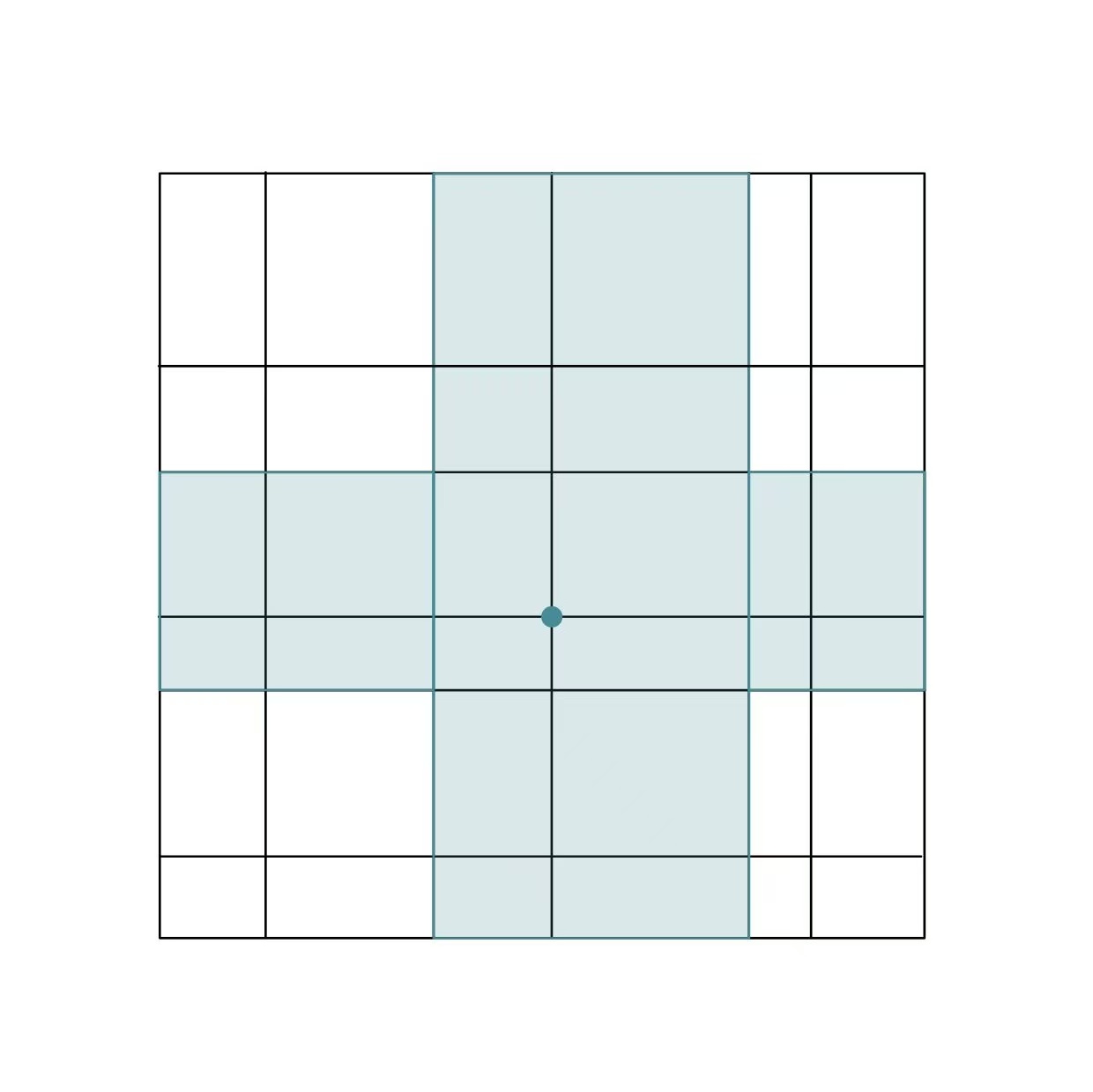}\\[-1.5em]
        \caption{An illustration of the two-dimensional case. The blue point in the middle is the newly added partition point \(\boldsymbol{c}\). Only part of rectangles are affected by the new point; they are divided into smaller rectangles.}
    \end{figure}
    
    It is easy to prove that if a rectangle is divided into smaller rectangles, the sum of the absolute values of the joint increments of \(f\) over the smaller rectangles is greater than or equal to the absolute value of the joint increment over the original rectangle. Therefore, the sum of the absolute values of the joint increments under partition \(\Delta_1\) is less than or equal to that under partition \(\Delta_2\). Thus (45) is proved.

    Next, prove
    \begin{equation}
        \bigvee_{\boldsymbol{a}}^{\boldsymbol{b}}(f)\geq\sum_{j\in\{0,1\}}\bigvee_{I_{1j}\times\dots\times I_{nj}}(f).
    \end{equation}
    By the definition of supremum, for any \(\varepsilon>0\), take a sufficiently large even number \(N\). There exists a partition
    \[a_i=x_{0i}<x_{1i}<\dots<x_{\frac{N}{2}i}=c_i<x_{(\frac{N}{2}+1)i}<\dots<x_{Ni}=b_i(i=1,\dots,n)\]
    dividing each \(I_{1j}\times\dots\times I_{nj}(j\in\{0,1\})\) into \((\frac{N}{2})^n\) small rectangles, and such that each \(I_{1j}\times\dots\times I_{nj}\) satisfies 
    \[\sum_{i=1}^{(\frac{N}{2})^n}\lvert\Delta_f^n I_i\rvert
    \geq\bigvee_{I_{1j}\times\dots\times I_{nj}}(f)-\frac{\varepsilon}{2^n}.\]
    This shows that there exists a partition of \([\boldsymbol{a},\boldsymbol{b}]\), such that
    \[\sum_{i=1}^{N^n}\lvert\Delta_f^n I_i\rvert\geq\sum_{j\in\{0,1\}}\bigvee_{I_{1j}\times\dots\times I_{nj}}(f)-\varepsilon.\]
    Thus,
    \[\bigvee_{\boldsymbol{a}}^{\boldsymbol{b}}(f)\geq\sum_{i=1}^{N^n}\lvert\Delta_f^n I_i\rvert\geq\sum_{j\in\{0,1\}}\bigvee_{I_{1j}\times\dots\times I_{nj}}(f)-\varepsilon.\]
    By the arbitrariness of \(\varepsilon>0\), we prove
    \[\bigvee_{\boldsymbol{a}}^{\boldsymbol{b}}(f)\geq\sum_{j\in\{0,1\}}\bigvee_{I_{1j}\times\dots\times I_{nj}}(f).\]
    
    Combining (45) and (46), we obtain
    \[\bigvee_{\boldsymbol{a}}^{\boldsymbol{b}}(f)=\sum_{j\in\{0,1\}}\bigvee_{I_{1j}\times\dots\times I_{nj}}(f).\]
\end{proof}
Now prove Theorem 1.4.
\begin{proof}
    Sufficiency. This follows from
    \[\bigvee_{\boldsymbol{a}}^{\boldsymbol{b}}(f)\leq\bigvee_{\boldsymbol{a}}^{\boldsymbol{b}}(g)+\bigvee_{\boldsymbol{a}}^{\boldsymbol{b}}(h)\leq\Delta_g^n [\boldsymbol{a},\boldsymbol{b}]+\Delta_h^n [\boldsymbol{a},\boldsymbol{b}].\]
    
    Necessity. Let
    \[g(\boldsymbol{x})=\bigvee_{\boldsymbol{a}}^{\boldsymbol{x}}(f),\]
    By the definition of total variation, obviously \(g\) is jointly monotone increasing on \([\boldsymbol{a},\boldsymbol{b}]\). Let
    \[h(\boldsymbol{x})=\bigvee_{\boldsymbol{a}}^{\boldsymbol{x}}(f)-f(\boldsymbol{x}),\]
    For any \(\boldsymbol{u}\leq\boldsymbol{v}\in[\boldsymbol{a},\boldsymbol{b}]\),
    \begin{equation}
        \begin{aligned}[b]
             \Delta_h^n[\boldsymbol{u},\boldsymbol{v}]
             &=\sum_{\epsilon_1,\dots,\epsilon_n\in\{0,1\}}(-1)^{n+\epsilon_1+\dots+\epsilon_n}\bigvee_{\boldsymbol{a}}^{\boldsymbol{u}+\boldsymbol{\epsilon}\circ(\boldsymbol{v}-\boldsymbol{u})}(f)\\
             &-\sum_{\epsilon_1,\dots,\epsilon_n\in\{0,1\}}(-1)^{n+\epsilon_1+\dots+\epsilon_n}f(\boldsymbol{u}+\boldsymbol{\epsilon}\circ(\boldsymbol{v}-\boldsymbol{u}))
        \end{aligned}
    \end{equation}
    By Lemma 2.3, the total variation is additive over domain, so (47) can be simplified to
    \[\Delta_h^n[\boldsymbol{u},\boldsymbol{v}]=\bigvee_{\boldsymbol{u}}^{\boldsymbol{v}}(f)-\Delta_f^n[\boldsymbol{u},\boldsymbol{v}]\geq0.\]
    By the definition of a jointly monotone increasing function, we know \(h(\boldsymbol{x})\) is jointly monotone increasing on \([\boldsymbol{a},\boldsymbol{b}]\). And because
    \[f(\boldsymbol{x})=g(\boldsymbol{x})-h(\boldsymbol{x}),\]
    it is proved that if \(f\) is of bounded variation on \([\boldsymbol{a},\boldsymbol{b}]\), it can be decomposed into the difference of two jointly monotone increasing functions.
\end{proof}
\subsection{Proof of Theorem 1.5}
\begin{proof}
    First, prove that if a function \(f\) defined on \([\boldsymbol{a},\boldsymbol{b}]\) is of bounded variation, then \(f^{(n)}(\boldsymbol{x})\) exists almost everywhere on \([\boldsymbol{a},\boldsymbol{b}]\).

    By Theorem 1.4, \(f\) can be decomposed into the difference of two jointly monotone increasing functions \(g,h\), where \(g=\bigvee\limits_{\boldsymbol{a}}^{\boldsymbol{b}}(f),h=\bigvee\limits_{\boldsymbol{a}}^{\boldsymbol{b}}(f)-f\), i.e.,
    \[f(\boldsymbol{x})=g(\boldsymbol{x})-h(\boldsymbol{x}),\boldsymbol{x}\in[\boldsymbol{a},\boldsymbol{b}].\]
    By Theorem 1.1, \(g^{(n)}\) and \(h^{(n)}\)exist almost everywhere. Then we have
    \begin{equation}
        \varliminf_{k\to0^+}\frac{\Delta_{g}^n[\boldsymbol{x},\boldsymbol{x}+k]}{k^n}=\varlimsup_{k\to0^+}\frac{\Delta_{g}^n[\boldsymbol{x},\boldsymbol{x}+k]}{k^n},
    \end{equation}
    \begin{equation}
        \varliminf_{k\to0^+}\frac{\Delta_{h}^n[\boldsymbol{x},\boldsymbol{x}+k]}{k^n}=\varlimsup_{k\to0^+}\frac{\Delta_{h}^n[\boldsymbol{x},\boldsymbol{x}+k]}{k^n}.
    \end{equation}
    We now prove that \(f^{(n)}\) exists almost everywhere and equals \(g^{(n)}-h^{(n)}\). Because
    \begin{equation}
        \varliminf_{k\to0^+}\frac{\Delta_f^n[\boldsymbol{x},\boldsymbol{x}+k]}{k^n}=\varliminf_{k\to0^+}\frac{\Delta_{g}^n[\boldsymbol{x},\boldsymbol{x}+k]}{k^n}-\varlimsup_{k\to0^+}\frac{\Delta_{h}^n[\boldsymbol{x},\boldsymbol{x}+k]}{k^n},
    \end{equation}        
    \begin{equation}
        \varlimsup_{k\to0^+}\frac{\Delta_f^n[\boldsymbol{x},\boldsymbol{x}+k]}{k^n}=\varlimsup_{k\to0^+}\frac{\Delta_{g}^n[\boldsymbol{x},\boldsymbol{x}+k]}{k^n}-\varliminf_{k\to0^+}\frac{\Delta_{h}^n[\boldsymbol{x},\boldsymbol{x}+k]}{k^n}.
    \end{equation}
    Therefore, from (48), (49), (50), and (51), we get
    \begin{align*}
        f^{(n)}(\boldsymbol{x})&=\lim_{k\to0^+}\frac{\Delta_f^n[\boldsymbol{x},\boldsymbol{x}+k]}{k^n}\\
        &=\lim_{k\to0^+}\frac{\Delta_{g}^n[\boldsymbol{x},\boldsymbol{x}+k]}{k^n}-\lim_{k\to0^+}\frac{\Delta_{h}^n[\boldsymbol{x},\boldsymbol{x}+k]}{k^n}\\
        &=g^{(n)}(\boldsymbol{x})-h^{(n)}(\boldsymbol{x}).
    \end{align*}
    Thus, it is proved that \(f^{(n)}\) exists almost everywhere on \([\boldsymbol{a},\boldsymbol{b}]\) and equals \(g^{(n)}-h^{(n)}\).
    
    Since \(g,h\in L[\boldsymbol{a},\boldsymbol{b}]\), hence \(f^{(n)}\in L[\boldsymbol{a},\boldsymbol{b}]\), i.e., \(f^{(n)}\) is integrable on \([\boldsymbol{a},\boldsymbol{b}]\).

    Now prove inequality (7). By Theorem 1.1, we have
    \begin{equation}
        \begin{aligned}[b]
            \int_{\boldsymbol{a}}^{\boldsymbol{b}}g^{(n)}(\boldsymbol{x})\diff \boldsymbol{x}&\leq\Delta_g^n[\boldsymbol{a},\boldsymbol{b}]\\
            &=\sum_{\epsilon_1,\dots,\epsilon_n\in\{0,1\}}(-1)^{n+\epsilon_1+\dots+\epsilon_n}\bigvee_{\boldsymbol{a}}^{\boldsymbol{a}+\boldsymbol{\epsilon}\circ(\boldsymbol{b}-\boldsymbol{a})}(f)\\
            &=0+\bigvee_{\boldsymbol{a}}^{\boldsymbol{b}}(f).    
        \end{aligned}
    \end{equation}
    We now prove that for almost every \(\forall\boldsymbol{x}\in[\boldsymbol{a},\boldsymbol{b}]\), \(f^{(n)},g^{(n)}\) satisfy
    \begin{equation}
        \lvert f^{(n)}(\boldsymbol{x})\rvert\leq g^{(n)}(\boldsymbol{x}).
    \end{equation}
    This follows from
    \begin{align*}
        f^{(n)}(\boldsymbol{x})&=\lim_{N\to+\infty}N^n\Delta_f^n[\boldsymbol{x},\boldsymbol{x}+\frac{1}{N}]\\
        &\leq\lim_{N\to+\infty}N^n\bigvee_{\boldsymbol{x}}^{\boldsymbol{x}+\frac{1}{N}}(f)\\
        &=g^{(n)}(\boldsymbol{x}).
    \end{align*}

    Therefore, from (52) and (53), we can prove
    \[\int_{\boldsymbol{a}}^{\boldsymbol{b}}\lvert f^{(n)}(\boldsymbol{x})\rvert\diff \boldsymbol{x}\leq\int_{\boldsymbol{a}}^{\boldsymbol{b}}g^{(n)}(\boldsymbol{x})\diff \boldsymbol{x}=\bigvee_{\boldsymbol{a}}^{\boldsymbol{b}}(f).\]
\end{proof}

\subsection{Proof of Corollary 1.6}
Before proving Corollary 1.6, we first give the following lemma.
\begin{lemma}
    Let \(\boldsymbol{a},\boldsymbol{b}\in\mathbb{R}^n\), \(f\) be a function defined on \([\boldsymbol{a},\boldsymbol{b}]\). Let \(g(\boldsymbol{x})=\bigvee\limits_{\boldsymbol{a}}^{\boldsymbol{x}}(f)\). \(f\) is absolutely continuous on \([\boldsymbol{a},\boldsymbol{b}]\) if and only if \(g\) is absolutely continuous on \([\boldsymbol{a},\boldsymbol{b}]\).
\end{lemma}
\begin{proof}
    Necessity. Because \(f\) is absolutely continuous on the domain, for any \(\varepsilon_1>0\), there exists \(\delta_1>0\) such that for any finite collection of closed rectangles \(I_i\subset[\boldsymbol{a},\boldsymbol{b}](i=1,\dots,N_1)\) with disjoint interiors, as long as \(\sum\limits_{i=1}^{N_1}mI_i<\delta\), we have
    \(\sum\limits_{i=1}^{N_1}\lvert\Delta_f^nI_i\rvert<\varepsilon_1\).

    Next, prove \(g\) is absolutely continuous on the domain. For any \(\varepsilon_2>0\), take \(\varepsilon_1=\frac{1}{2}\varepsilon_2\), obtaining the corresponding \(\delta_1\). Take \(\delta_2=\delta_1\). For any finite collection of closed rectangles \(I_k\subset[\boldsymbol{a},\boldsymbol{b}](k=1,\dots,N_2)\) with disjoint interiors, satisfying \(\sum\limits_{k=1}^{N_2}I_k\leq\delta_2\), we have
    \begin{equation}
        \sum_{k=1}^{N_2}\lvert\Delta_g^nI_k\rvert=\sum_{k=1}^{N_2}\bigvee_{I_k}(f).
    \end{equation}
    By the definition of total variation, for \(\{I_k\}_{k=1}^{N_2}\), there exists a partition \(\Delta_{\alpha}\) such that
    \begin{equation}
        \sum_{k=1}^{N_2}\bigvee_{I_k}(f)\leq\sum_{\alpha}\lvert\Delta_f^nI_{\alpha}\rvert+\frac{1}{2}\varepsilon_2.
    \end{equation}
    Furthermore, since the sub-rectangles \(I_\alpha\) form a partition of \(\{I_k\}_{k=1}^{N_2}\), we have
    \[\sum_\alpha mI_{\alpha}=\sum_{k=1}^{N_2}mI_k\leq\delta_2=\delta_1.\]
    Therefore, by the absolute continuity of \(f\),
    \begin{equation}
        \sum_\alpha\lvert\Delta_f^nI_{\alpha}\rvert\leq\varepsilon_1=\frac{1}{2}\varepsilon_2.
    \end{equation}
    Combining (54), (55), and (56), we get
    \[\sum_{k=1}^{N_2}\lvert\Delta_g^nI_k\rvert\leq\varepsilon_2.\]
    Therefore, by the definition of absolute continuity, \(g\) is absolutely continuous on the domain.

    Sufficiency. If \(g\) is absolutely continuous on the domain, then for any \(\varepsilon_1>0\), there exists \(\delta_1>0\) such that for any finite collection of closed rectangles \(I_i\subset[\boldsymbol{a},\boldsymbol{b}](i=1,\dots,N_1)\) with disjoint interiors, as long as \(\sum\limits_{i=1}^{N_1}mI_i<\delta\), we have
    \begin{equation}
        \sum\limits_{i=1}^{N_1}\lvert\Delta_g^nI_i\rvert=\sum\limits_{i=1}^{N_1}\bigvee_{I_i}(f)<\varepsilon_1.
    \end{equation}

    Next, prove \(f\) is absolutely continuous on the domain. For any \(\varepsilon_2>0\), take \(\varepsilon_1=\varepsilon_2\), obtaining the corresponding \(\delta_1\). Take \(\delta_2=\delta_1\). For any finite collection of closed rectangles \(I_k\subset[\boldsymbol{a},\boldsymbol{b}](k=1,\dots,N_2)\) with disjoint interiors, satisfying \(\sum\limits_{k=1}^{N_2}I_k\leq\delta_2\), we have
    \begin{equation}
        \sum_{k=1}^{N_2}\lvert\Delta_f^nI_k\rvert\leq\sum_{k=1}^{N_2}\bigvee_{I_k}(f).
    \end{equation}
    By (57), (58) can be bounded by:
    \[\sum_{k=1}^{N_2}\lvert\Delta_f^nI_k\rvert\leq\sum_{k=1}^{N_2}\bigvee_{I_k}(f)<\varepsilon_1=\varepsilon_2.\]
    Therefore, by the definition of absolute continuity, \(f\) is absolutely continuous on the domain.
\end{proof}

Now prove Corollary 1.6.
\begin{proof}
    Necessity. Because \(f\) is of bounded variation on \([\boldsymbol{a},\boldsymbol{b}]\), by Theorem 1.4, \(f\) can be decomposed into the difference of two jointly monotone increasing functions, i.e., \(f(\boldsymbol{x})=g(\boldsymbol{x})-h(\boldsymbol{x})\), where \(g(\boldsymbol{x})=\bigvee\limits_{\boldsymbol{a}}^{\boldsymbol{x}}(f),h(\boldsymbol{x})=\bigvee\limits_{\boldsymbol{a}}^{\boldsymbol{x}}(f)-f(\boldsymbol{x})\).
    
    Next, prove
    \begin{equation}
        \int_{\boldsymbol{a}}^{\boldsymbol{b}}\lvert f^{(n)}(\boldsymbol{x})\rvert\diff\boldsymbol{x}\geq\bigvee_{\boldsymbol{a}}^{\boldsymbol{b}}(f).
    \end{equation}
    
    Because \(f\) is absolutely continuous on \([\boldsymbol{a},\boldsymbol{b}]\), by Lemma 2.4, \(g\) is absolutely continuous on \([\boldsymbol{a},\boldsymbol{b}]\). 
    It is easy to prove that \(h(\boldsymbol{x})=g(\boldsymbol{x})-f(\boldsymbol{x})\) is absolutely continuous on \([\boldsymbol{a},\boldsymbol{b}]\). 
    Since \(g,h\) are jointly monotone increasing and absolutely continuous, by Theorem 1.1, for any rectangle \(I\subset[\boldsymbol{a},\boldsymbol{b}]\), we have
    \begin{equation}
        \int_Ig^{(n)}(\boldsymbol{x})\diff\boldsymbol{x}=\bigvee_I(g)=\Delta_g^nI=\bigvee_I(f),
    \end{equation}
    \begin{equation}
        \int_Ih^{(n)}(\boldsymbol{x})\diff\boldsymbol{x}=\bigvee_I(h)=\Delta_h^nI=\bigvee_I(f)-\Delta_f^nI.
    \end{equation}
    Therefore, from (60) and (61), it follows that
    \begin{equation}
         \int_If^{(n)}(\boldsymbol{x})\diff\boldsymbol{x}=\int_Ig^{(n)}(\boldsymbol{x})\diff\boldsymbol{x}-\int_Ih^{(n)}(\boldsymbol{x})\diff\boldsymbol{x}=\Delta_f^nI.  
    \end{equation}
    
    By the definition of total variation, for any \(\varepsilon>0\), there exists a partition \(\Delta_\alpha\), such that
    \begin{equation}
        \bigvee_{\boldsymbol{a}}^{\boldsymbol{b}}(f)\leq\sum_\alpha\lvert \Delta_f^nI_\alpha\rvert+\varepsilon.
    \end{equation}
    Therefore, from (62), we have
    \begin{align*}
        \bigvee_{\boldsymbol{a}}^{\boldsymbol{b}}(f)&\leq\sum_\alpha\lvert \Delta_f^nI_\alpha\rvert+\varepsilon\\
        &=\sum_\alpha\lvert \int_{I_\alpha}f^{(n)}(\boldsymbol{x})\diff\boldsymbol{x}\rvert+\varepsilon\\
        &\leq\sum_\alpha\int_{I_\alpha}\lvert f^{(n)}(\boldsymbol{x})\rvert\diff\boldsymbol{x}+\varepsilon.
    \end{align*}
    Since \(\varepsilon\) is arbitrary, the bound is proved.

    Since \(f\) is of bounded variation on \([\boldsymbol{a},\boldsymbol{b}]\), inequality (7) holds. Combining (7) and (59), we prove that if \(f\) is absolutely continuous on \([\boldsymbol{a},\boldsymbol{b}]\), then
    \[\int_{\boldsymbol{a}}^{\boldsymbol{b}}\lvert f^{(n)}(\boldsymbol{x})\rvert\diff \boldsymbol{x}=\bigvee_{\boldsymbol{a}}^{\boldsymbol{b}}(f).\]

    Sufficiency. From the condition of the corollary, \(f\) is of bounded variation on \([\boldsymbol{a},\boldsymbol{b}]\), so \(\bigvee\limits_{\boldsymbol{a}}^{\boldsymbol{b}}(f)\) is finite. Furthermore, since
    \[\int_{\boldsymbol{a}}^{\boldsymbol{b}}\lvert f^{(n)}(\boldsymbol{x})\rvert\diff \boldsymbol{x}=\bigvee_{\boldsymbol{a}}^{\boldsymbol{b}}(f),\]
    it follows that \(\lvert f^{(n)}(\boldsymbol{x})\rvert\) is integrable on \([\boldsymbol{a},\boldsymbol{b}]\). 
    By the absolute continuity of integrable functions, and by
    \[\bigvee_{\boldsymbol{a}}^{\boldsymbol{b}}(f)=\int_{\boldsymbol{a}}^{\boldsymbol{b}}\lvert f^{(n)}(\boldsymbol{x})\rvert\diff \boldsymbol{x},\]
    it can be shown that \(g(\boldsymbol{x})=\bigvee\limits_{\boldsymbol{a}}^{\boldsymbol{x}}(f)\) is absolutely continuous on \([\boldsymbol{a},\boldsymbol{b}]\). 
    From the definition of absolute continuity, it is easy to prove that \(f(\boldsymbol{x})\) is absolutely continuous on \([\boldsymbol{a},\boldsymbol{b}]\).
\end{proof}

\subsection{Proof of Theorem 1.7}
Before proving Theorem 1.7, we first give a lemma.
\begin{lemma}{(Riesz)}
    Let \(\{\varphi_{i}\}\) be a sequence of nonnegative jointly monotone increasing functions defined on the rectangle \([\boldsymbol{a},\boldsymbol{b}]\). If for \(\forall\boldsymbol{x}\in[\boldsymbol{a},\boldsymbol{b}]\), \(\sum\limits_{i=1}^{+\infty}\lvert\varphi_{i}(\boldsymbol{x})\rvert<+\infty\), then \(\lim\limits_{i\to +\infty}\varphi_{i}^{(n)}(x)=0\) holds almost everywhere on \([\boldsymbol{a},\boldsymbol{b}]\).   
\end{lemma}
\begin{proof}
    Let \(\varphi(\boldsymbol{x})=\sum\limits_{i=1}^{+\infty}\varphi_{i}(\boldsymbol{x})\). By the conditions of the lemma, \(\varphi(\boldsymbol{x})\) is finite on \([\boldsymbol{a},\boldsymbol{b}]\), so by the additivity of convergent series, it is easy to prove that \(\varphi(\boldsymbol{x})\) is jointly monotone increasing on \([\boldsymbol{a},\boldsymbol{b}]\).
    
    Because \(\varphi,\{\varphi_i\}\) are all jointly monotone increasing functions, by Theorem 1.1, \(\varphi^{(n)}(\boldsymbol{x})\), the series \(\sum\varphi_{i}^{(n)}(\boldsymbol{x})\) converges on \([\boldsymbol{a},\boldsymbol{b}]\setminus E\), where \(m(E)=0\). 
   
    By the additivity of convergent series, it is easy to prove that
    \[\varphi(\boldsymbol{x})-\sum_{i=1}^{n}\varphi_{i}(\boldsymbol{x})=\sum_{i=n+1}^{+\infty}\varphi_{i}(\boldsymbol{x})\]
     is jointly monotone increasing on \([\boldsymbol{a},\boldsymbol{b}]\), so for \(\boldsymbol{x}\in[\boldsymbol{a},\boldsymbol{b}]\setminus E\), we have 
    \[\left(\varphi(\boldsymbol{x})-\sum_{i=1}^{n}\varphi_{i}(\boldsymbol{x})\right)^{(n)}=\varphi^{(n)}(\boldsymbol{x})-\sum_{i=1}^{n}\varphi_{i}^{(n)}(\boldsymbol{x})\geq 0,\]
    i.e., \(\sum\limits_{i=1}^{n}\varphi_{i}^{(n)}(\boldsymbol{x})\leq\varphi^{(n)}(\boldsymbol{x})<+\infty\). Since the series the series \(\sum\varphi_{i}^{(n)}(\boldsymbol{x})\)  of nonnegative terms converges almost everywhere on \([\boldsymbol{a},\boldsymbol{b}]\), it can be proved that \(\lim\limits_{i\to+\infty}\varphi_{i}^{(n)}(\boldsymbol{x})=0\) holds almost everywhere on \([\boldsymbol{a},\boldsymbol{b}]\).
\end{proof}

Now prove Theorem 1.7.
\begin{proof}
    Let
    \[f(\boldsymbol{x})=\sum_{i=1}^{+\infty}f_{i}(\boldsymbol{x}),\]
    Let
    \[\varphi_{N}(\boldsymbol{x}) = f(\boldsymbol{x})-\sum_{i=1}^{N}f_{i}(x)=\sum_{i=N+1}^{+\infty}f_{i}(\boldsymbol{x})(N=1,2,\cdots).\]
    It is easy to prove that \(\varphi_N\) is finite and jointly monotone increasing on \([\boldsymbol{a},\boldsymbol{b}]\), therefore \(\varphi_N^{(n)}(\boldsymbol{x})\) exists and is finite on \([\boldsymbol{a},\boldsymbol{b}]\setminus E_1\), where \(m(E_1)=0\).
    
    Because
    \[\varphi_{N}(\boldsymbol{x})-\varphi_{N+1}(\boldsymbol{x})=f_{N+1}(\boldsymbol{x}),\]
    for \(\boldsymbol{x}\in[\boldsymbol{a},\boldsymbol{b}]\setminus E_1\), it follows that
    \[\varphi_{N}^{(n)}(\boldsymbol{x})-\varphi_{N+1}^{(n)}(\boldsymbol{x})=f_{N+1}^{(n)}(\boldsymbol{x})\geq 0.\]
    This shows that the sequence \(\{\varphi_{N}^{(n)}(x)\}\) is decreasing and converges on \([\boldsymbol{a},\boldsymbol{b}]\setminus E_1\).
    
    Since \(\lim\limits_{N\to+\infty}\varphi_{N}(\boldsymbol{x})=0\), there exists a subsequence \(\{\varphi_{k_{N}}(\boldsymbol{x})\}\) such that \(\varphi_{k_{N}}(\boldsymbol{x})<\frac{1}{2^{N}}\), then
    \[\sum_{N=1}^{+\infty}\varphi_{k_{N}}(\boldsymbol{x})<+\infty.\]
    By Lemma 2.5, it can be proved that \(\lim\limits_{N\to +\infty}\varphi_{k_{N}}^{(n)}(\boldsymbol{x})=0\) holds on \([\boldsymbol{a},\boldsymbol{b}]\setminus E_2\), where \(m(E_2)=0\). Therefore, for \(\boldsymbol{x}\in[\boldsymbol{a},\boldsymbol{b}]\setminus (E_1\cup E_2)\), we have
    \[\lim_{N\to+\infty}\left(f^{(n)}(\boldsymbol{x})-\sum_{i=1}^{N}f_{i}^{(n)}(\boldsymbol{x})\right)=0,\]
    i.e.,
    \[\left(\sum_{i=1}^{\infty} f_{i}\right)^{(n)}(\boldsymbol{x})= \sum_{i=1}^{\infty} f_{i}^{(n)}(\boldsymbol{x}).\]
\end{proof}
\bibliographystyle{amsplain}

\appendix
\section{}
\subsection{Proof of Corollary 1.8}
Before proving Corollary 1.8, we first give the following lemma.
\begin{lemma}
    If \(f\in L[\boldsymbol{a},\boldsymbol{b}]\), let
    \[F_h(\boldsymbol{x})=\frac{1}{h^n}\int_{\boldsymbol{x}}^{\boldsymbol{x}+h}f\diff \boldsymbol{x},\]
    then
    \[
        \lim_{h\to0^+}\int_{\boldsymbol{a}}^{\boldsymbol{b}}\lvert F_h(\boldsymbol{x})-f(\boldsymbol{x})\rvert\diff \boldsymbol{x}=0.
    \]
\end{lemma}

\begin{proof}
    \[F_h(\boldsymbol{x})-f(\boldsymbol{x})=\frac{1}{h^n}\int_0^h\dots\int_0^h f(x_1+p_1,\dots,x_n+p_n)-f(\boldsymbol{x})\diff p_1\dots\diff p_n.\]
    Then
    \begin{align*}
        &\int_{\boldsymbol{a}}^{\boldsymbol{b}}\lvert F_h(\boldsymbol{x})-f(\boldsymbol{x})\rvert\diff \boldsymbol{x}\\
        &=\frac{1}{h^n}\int_{\boldsymbol{a}}^{\boldsymbol{b}}\int_0^h\dots\int_0^h\lvert f(x_1+p_1,\dots,x_n+p_n)-f(\boldsymbol{x})\rvert\diff p_1\dots\diff p_n\diff \boldsymbol{x}.
    \end{align*}
    
    We now prove that \(f(x_1+p_1,\dots,x_n+p_n)-f(\boldsymbol{x})\) is measurable with respect to \(x_i,p_i\). Take \((x_1,p_1)\) as an example. Focus only on the part of the integrand related to \(x_1,p_1\), i.e., focus on \(f(x_1+p_1)\).
    By Tonelli's theorem, for \(a.e.(x_2+p_2,\dots,x_n+p_n)\in\mathbb{R}^{n-1},f|_{(x_2+p_2,\dots,x_n+p_n)}(x_1+p_1)\) as a function of \(x_1,p_1\) is measurable on \(\mathbb{R}\). Consider only the case where \(f|_{(x_2+p_2,\dots,x_n+p_n)}(x_1+p_1)\) is measurable as a function of \(x_1,p_1\) on \(\mathbb{R}\).
    
    Let it be \(f\circ h(x_1,p_1)\), where \(h(x_1,p_1)=x_1+p_1\). Let \(G\) be any open set in \(\mathbb{R}\), then \(f^{-1}(G)\) is a measurable set in \(\mathbb{R}\), denoted as \(G_{\delta}-N(mN=0)\).
    \[h^{-1}(f^{-1}(G))=h^{-1}(G_{\delta})+h^{-1}(N).\]
    Because \(h\) is a continuous function, \(h^{-1}(G_{\delta})\) is an open set in \(\mathbb{R}\). 
    We now prove that \(h^{-1}(N)\) is a null set in \(\mathbb{R}^2\).
    \(x_1+p_1=N\), i.e., \(p_1=N-x_1\), thus 
    \[m(x_1\times p_1)=m(\mathbb{R}\times N)=0.\]
    Since null sets are measurable, \(h^{-1}(N)\) is a measurable set in \(\mathbb{R}^2\), thus \(f\circ h,h^{-1}\circ f^{-1}\) map measurable sets to measurable sets, proving that \(f(x_1+p_1)\) is a measurable function of \(x_1,p_1\).
    
    It can be similarly proved that \(f(x_1+p_1,\dots,x_n+p_n)-f(\boldsymbol{x})\) is measurable with respect to \(x_i,p_j(i,j=1\dots,n)\).

    Since \(f(x_1+p_1,\dots,x_n+p_n)-f(\boldsymbol{x})\) is measurable with respect to \(x_i,p_j;\forall x_i,x_j(i\neq j);\forall p_i,p_j(i\neq j)\), by Tonelli's theorem, the order of the n-fold integral can be changed:
    \begin{align*}
        &\frac{1}{h^n}\int_{\boldsymbol{0}}^{\boldsymbol{a}}\int_0^h\dots\int_0^h\lvert f(x_1+p_1,\dots,x_n+p_n)-f(\boldsymbol{x})\rvert\diff p_1\dots\diff p_n\diff \boldsymbol{x}\\
        &=\frac{1}{h^n}\int_0^h\dots\int_0^h\diff p_1\dots\diff p_n\int_{\boldsymbol{a}}^{\boldsymbol{b}}\lvert f(x_1+p_1,\dots,x_n+p_n)-f(\boldsymbol{x})\rvert\diff \boldsymbol{x}\\
        &\leq\frac{1}{h^n}\int_0^h\dots\int_0^h\diff p_1\dots\diff p_n\int_{\mathbb{R}^n}\lvert f(x_1+p_1,\dots,x_n+p_n)-f(\boldsymbol{x})\rvert\diff \boldsymbol{x}.
    \end{align*}
    According to the average continuity of integrable functions, for any \(\varepsilon>0\), there exists \(\delta>0\), such that whenever \(0<t<\delta\), we have
    \[\int_{\mathbb{R}^n}\lvert f(\boldsymbol{x}+t)-f(\boldsymbol{x})\rvert\diff \boldsymbol{x}<\varepsilon.\]
    Then, if \(0<p_i<\delta\), we have
    \begin{align*}
        &\int_{\mathbb{R}^n}\lvert f(x_1+p_1,\dots,x_n+p_n)-f(\boldsymbol{x})\rvert\diff \boldsymbol{x}\\
        &=\int_{\mathbb{R}^n}\lvert f(x_1+p_1,\dots,x_n+p_n)-f(x_1,x_2+p_2\dots,x_n+p_n)\\
        &+f(x_1,x_2+p_2\dots,x_n+p_n)-f(x_1,x_2,x_3+p_3\dots,x_n+p_n)\\
        &+\dots\\
        &+f(x_1\dots,x_{n-1},x_n+p_n)-f(x_1,\dots,x_n)\rvert\diff \boldsymbol{x}\\
        &\leq \int_{\mathbb{R}^n}\lvert f(x_1+p_1,\dots,x_n+p_n)-f(x_1,x_2+p_2\dots,x_n+p_n)\rvert\diff \boldsymbol{x}\\
        &+\dots+\int_{\mathbb{R}^n}\lvert f(x_1\dots,x_{n-1},x_n+p_n)-f(x_1,\dots,x_n)\rvert\diff \boldsymbol{x}\\
        &< n\varepsilon.
    \end{align*}
    Thus, whenever \(0<h<\delta\),
    \begin{align*}
        &\frac{1}{h^n}\int_{\boldsymbol{a}}^{\boldsymbol{b}}\int_0^h\dots\int_0^h\lvert f(x_1+p_1,\dots,x_n+p_n)-f(\boldsymbol{x})\rvert\diff p_1\dots\diff p_n\diff \boldsymbol{x}\\
        &\leq\frac{1}{h^n}\int_0^h\dots\int_0^h\diff p_1\dots\diff p_n\int_{\boldsymbol{a}}^{\boldsymbol{b}}\lvert f(x_1+p_1,\dots,x_n+p_n)-f(\boldsymbol{x})\rvert\diff \boldsymbol{x}\\
        &<\frac{1}{h^n}\int_0^h\dots\int_0^h\diff p_1\dots\diff p_n\cdot n\varepsilon\\
        &=n\varepsilon.
    \end{align*}
    This proves
    \[\lim_{h\to0^+}\int_{\boldsymbol{a}}^{\boldsymbol{b}}\lvert F_h(\boldsymbol{x})-f(\boldsymbol{x})\rvert\diff \boldsymbol{x}=0.\]
\end{proof}
Now prove Corollary 1.8.
\begin{proof}
    \(\int_{\boldsymbol{a}}^{\boldsymbol{x}} f(\boldsymbol{x})\diff \boldsymbol{x}\)can be decomposed as the difference of two jointly monotone increasing functions, 
    \[\int_{\boldsymbol{a}}^{\boldsymbol{x}} f\diff \boldsymbol{x}=\int_{\boldsymbol{a}}^{\boldsymbol{x}} f^+\diff \boldsymbol{x}-\int_{\boldsymbol{a}}^{\boldsymbol{x}} f^-\diff \boldsymbol{x},\]
    therefore, by Theorem 1.1, \(\int_{\boldsymbol{a}}^{\boldsymbol{x}}f(\boldsymbol{x})\diff \boldsymbol{x}\)has a joint derivative almost everywhere. By Lemma A.1, we have 
    \[\lim_{k\to+\infty}\int_{\boldsymbol{a}}^{\boldsymbol{b}}\lvert F_{\frac{1}{k}}(\boldsymbol{x})-f(\boldsymbol{x})\rvert\diff \boldsymbol{x}=0.\]
    By Chebyshev's inequality, for any \(\sigma\), we have
    \[\sigma\cdot m\left\{\boldsymbol{x}\mid \boldsymbol{x}\in[\boldsymbol{a},\boldsymbol{b}],\lvert F_{\frac{1}{k}}(\boldsymbol{x})-f(\boldsymbol{x})\rvert\geq\sigma\right\}\rightarrow0,(k\rightarrow+\infty),\]
    thus \(F_{\frac{1}{k}}(\boldsymbol{x})\) converges in measure to \(f(\boldsymbol{x})\). According to Riesz's theorem, there exists a subsequence \(k_l\) such that
    \[\lim_{l\to+\infty}\frac{1}{k_l}F_{\frac{1}{k_l}}(\boldsymbol{x})=f(\boldsymbol{x}),a.e.\boldsymbol{x}\in[\boldsymbol{a},\boldsymbol{b}],\]
    hence
    \[\lim_{h\to 0^+}F_h(\boldsymbol{x})=f(\boldsymbol{x}),a.e.\boldsymbol{x}\in[\boldsymbol{a},\boldsymbol{b}],\]
    i.e.,
    \[\left(\int_{\boldsymbol{a}}^{\boldsymbol{x}}f(\boldsymbol{x})\diff \boldsymbol{x}\right)^{(n)}(\boldsymbol{x})=f(\boldsymbol{x}).\]
\end{proof}
\subsection{Proof of Corollary 1.9}
Before proving Corollary 1.9, we first give the following two lemmas.
\begin{lemma}
    If \(f\) is absolutely continuous on \([\boldsymbol{a},\boldsymbol{b}]\subset\mathbb{R}^n\), then \(f\) is of bounded variation on \([\boldsymbol{a},\boldsymbol{b}]\).
\end{lemma}

\begin{proof}
    By the definition of absolute continuity, take \(\varepsilon=1\). Then there exists \(0<\delta<1\) such that for any finite collection of closed rectangles \(I_i\subset[\boldsymbol{a},\boldsymbol{b}](i=1,\dots,N)\) with disjoint interiors, as long as \(m\sum\limits_{i=1}^{N}I_i<\delta\), we have
    \[\sum_{i=1}^{N}\lvert\Delta_f^nI_i\rvert<\varepsilon=1.\]
    
    Next, prove \(f\) is of bounded variation on \([\boldsymbol{a},\boldsymbol{b}]\). Find a sufficiently large integer \(N\) such that for each \(i\), and partition each \(\frac{b_i-a_i}{N}<\delta<1\) into \([a_i,b_i]N\) equal subintervals. This partition divides \([\boldsymbol{a},\boldsymbol{b}]\) into \(N^n\) rectangles, each satisfying 
    \[mI_i<\delta^N<\delta.\]
    On each rectangle \(I_i\), no matter what partition \(\Delta_\alpha\) is used to divide \(I_i\) into \(\sum\limits_{\alpha} I_\alpha\), we always have \(m\sum\limits_{\alpha} I_\alpha<\delta\), so for any partition, \(\sum\limits_{\alpha}\lvert \Delta_f^nI_\alpha\rvert<\varepsilon\). By the definition of total variation, we have
    \[\bigvee_{I_i}(f)<\varepsilon=1.\]
    Thus,
    \[\bigvee_{\boldsymbol{a}}^{\boldsymbol{b}}(f)=\sum_{i=1}^{N^n}\bigvee_{I_i}(f)<N^n.\]
    This proves that \(f\) is of bounded variation on \([\boldsymbol{a},\boldsymbol{b}]\).
\end{proof}
\begin{lemma} 
    If \(f(\boldsymbol{x})\) is absolutely continuous on \([\boldsymbol{a},\boldsymbol{b}]\), and for almost every \(\boldsymbol{x}\in [\boldsymbol{a},\boldsymbol{b}],f^{(n)}(\boldsymbol{x})=0\). Then \(\Delta_f^n[\boldsymbol{a},\boldsymbol{b}]=0\).
\end{lemma}

\begin{proof}
    Proof contradiction. Suppose not, then \(\Delta_f^n[\boldsymbol{a},\boldsymbol{b}]\neq0\). Let \(\varepsilon_0=\frac{1}{2}\lvert\Delta_f^n[\boldsymbol{a},\boldsymbol{b}]\rvert\). Define the set 
    \[E=\left\{\boldsymbol{x}\mid \boldsymbol{x}\in(\boldsymbol{a},\boldsymbol{b}), f^{(n)}(\boldsymbol{x})=0\right\},\]
    then for any \(\boldsymbol{x}\in E\), since \(f^{(n)}(\boldsymbol{x})=0\), for any \(r>0\), as long as \(h\) is sufficiently small, we have
    \[\lvert\Delta_f^n[\boldsymbol{x},\boldsymbol{x}+h]\rvert<rh^n,\]
    and \([\boldsymbol{x},\boldsymbol{x}+h]\subset [\boldsymbol{a},\boldsymbol{b}]\).
    The collection of all such \([\boldsymbol{x},\boldsymbol{x}+h]\) forms a Vitali covering of \(E\).
    
    Since \(f(\boldsymbol{x})\in AC[\boldsymbol{a},\boldsymbol{b}]\), there exists \(\delta_0\), such that if
    \[\sum_{i=1}^NmI_i<\delta_0,\]
    then
    \[\sum_{i=1}^N\lvert\Delta_f^n I_i\rvert<\varepsilon_0.\]

    Take \(\varepsilon_1=\delta_0\). By the Vitali covering theorem, there exist n-dimensional rectangles
    \[[\boldsymbol{x_1},\boldsymbol{x_1}+h_1],\dots,[\boldsymbol{x_N},\boldsymbol{x_N}+h_N]\subset [\boldsymbol{a},\boldsymbol{b}]\]
    with disjoint interiors. 
    Let their union be\(G\), 
    such that
    \[m\left(E\backslash G\right)=m\left([\boldsymbol{a},\boldsymbol{b}]\backslash G\right)<\varepsilon_1.\]
    Reorder the points \(x_{1i},x_{1i}+h_{1i},\dots,x_{Ni},x_{Ni}+h_{Ni}(i=1,\dots,n)\) in increasing order into \(p_{1i},\dots,p_{2Ni}(i=1,\dots,n)\). All the points \(x_i=p_{ji}(i=1,\dots,n,j=1\dots,2N)\) partition \([\boldsymbol{a},\boldsymbol{b}]\) into \((2N+1)^n\)  rectangles \(I_i\), then
    \begin{equation}
        \begin{aligned}[b]
        \lvert\Delta_f^n[\boldsymbol{a},\boldsymbol{b}]\rvert&\leq\sum_{i=1}^N\lvert\Delta_f^n[\boldsymbol{x_i},\boldsymbol{x_i}+h_i]
        \rvert+\sum_{[\boldsymbol{a},\boldsymbol{b}]\backslash G}\lvert\Delta_f^nI_i\rvert\\
        &<r\sum_{i=1}^Nh_i^n+\sum_{[\boldsymbol{a},\boldsymbol{b}]\backslash G}\lvert\Delta_f^nI_i\rvert.
        \end{aligned}
    \end{equation}
    
    Because \(f\in AC[\boldsymbol{a},\boldsymbol{b}]\), and since \(m\left([\boldsymbol{a},\boldsymbol{b}]\backslash G\right)<\varepsilon_1=\delta_0\), we have 
    \[\sum_{[\boldsymbol{a},\boldsymbol{b}]\backslash G}\lvert\Delta_f^nI_i\rvert<\varepsilon_0.\]
    Choose \(r>0\) such that \(r\sum\limits_{i=1}^Nh_i^n<\epsilon_0\). Then (64) becomes 
    \[2\varepsilon_0<2\varepsilon_0.\]
    This is a contradiction. Therefore, \(\Delta_f^n[\boldsymbol{a},\boldsymbol{b}]=0\).
\end{proof}

Now prove Corollary 1.9.
\begin{proof}
    By Lemma A.2, since \(f(\boldsymbol{x})\in AC[\boldsymbol{a},\boldsymbol{b}]\), we have \(f(\boldsymbol{x})\in BV[\boldsymbol{a},\boldsymbol{b}]\).
    
    By Theorem 1.4, if \(f(\boldsymbol{x})\in BV[\boldsymbol{a},\boldsymbol{b}]\), then \(f\) can be decomposed into the difference of two jointly monotone increasing functions.
    
    By Theorem 1.1, \(f^{(n)}(\boldsymbol{x})\) exists almost everywhere on \([\boldsymbol{a},\boldsymbol{b}]\) and is integrable.
    By Corollary 1.8, since \(f^{(n)}(\boldsymbol{x})\) is integrable on \([\boldsymbol{a},\boldsymbol{b}]\), we have
    \[\left(\int_{\boldsymbol{a}}^{\boldsymbol{x}} f^{(n)}\diff \boldsymbol{x}\right)^{(n)}(\boldsymbol{x})=f^{(n)}(\boldsymbol{x}).\]
   
    Because the n-dimensional integral of an integrable function is absolutely continuous, \(\int_{\boldsymbol{a}}^{\boldsymbol{x}} f^{(n)}\diff \boldsymbol{x}\) is absolutely continuous on \(\mathbb{R}^n\).
    
    Since \(f-\int_{\boldsymbol{a}}^{\boldsymbol{x}} f^{(n)}\diff \boldsymbol{x}\) is absolutely continuous on \([\boldsymbol{a},\boldsymbol{b}]\), and \(\left(f-\int_{\boldsymbol{a}}^{\boldsymbol{x}} f^{(n)}\diff \boldsymbol{x}\right)^{(n)}=0\) almost everywhere on \([\boldsymbol{a},\boldsymbol{b}]\), according to Lemma A.3, we have
    \[\Delta_{f-\int_{\boldsymbol{a}}^{\boldsymbol{x}} f^{(n)}\diff \boldsymbol{x}}^n[\boldsymbol{a},\boldsymbol{b}]=0.\]
    Simplifying yields
    \[\Delta_f^n[\boldsymbol{a},\boldsymbol{b}]=\int_{\boldsymbol{a}}^{\boldsymbol{b}}f^{(n)}\diff \boldsymbol{x}.\]
\end{proof}
\end{document}